\documentclass{article}

\usepackage{arxiv}

\usepackage[utf8]{inputenc} 
\usepackage[T1]{fontenc}    
\usepackage{hyperref}       
\usepackage{url}            
\usepackage{booktabs}       
\usepackage{amsfonts}       
\usepackage{nicefrac}       
\usepackage{microtype}      
\usepackage{lipsum}         
\usepackage{graphicx}
\usepackage{doi}
\usepackage{graphicx}
\usepackage{amssymb}  
\usepackage{amsmath}
\usepackage{amsthm} 
\usepackage{pgfplots}
\usepackage{cleveref}       

\newtheorem{theorem}{Theorem}
\newtheorem{problem}{Problem}

\title{Optimal Motions of an Elastic Structure under Finite-Dimensional Distributed Control \thanks{The study has been done under financial support of
	the Russian Science Foundation (grant 21-11-00151).}
}

\author{ \href{https://orcid.org/ 0000-0001-6526-6246}{\includegraphics[scale=0.06]{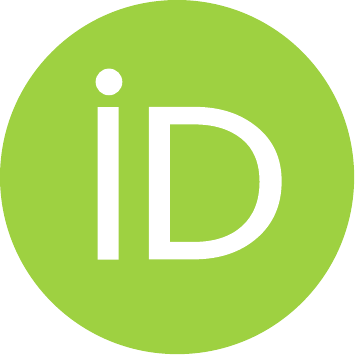}\hspace{1mm}Georgy Kostin}
 \\	Ishlinsky Institute for Problems in Mechanics RAS\\ 
	Moscow, Russia \\	
		\texttt{kostin@ipmnet.ru} \\
	\And
	\href{https://orcid.org/0000-0001-7210-5395}{\includegraphics[scale=0.06]{orcid.pdf}\hspace{1mm}Alexander Gavrikov} \\
	Department of Mathematics\\
	Penn State University\\
	State College, USA \\
	\texttt{avg6113@psu.edu} \\
}

\renewcommand{\shorttitle}{Optimal Motions of an Elastic Structure}

\begin{document}
\def\ds{\displaystyle}
\maketitle

\begin{abstract}
An optimal control problem for  longitudinal  motions of {a thin} elastic rod is considered.
We suppose that a normal force, which changes piecewise constantly along the rod's length, is applied to the cross section {so} that the positions of force jumps are equidistantly placed along the length. Additionally,   external loads act at the rod ends. These {distributed  force} and boundary loads  are considered as control functions of the dynamic system. 
Given  initial and  terminal states at fixed time instants, the problem is to  minimize the mean mechanical energy stored in the rod during {its} motion. We replace the {classical} wave equation with a variational problem solved via traveling waves defined on a special {time-space}  mesh. {For  a uniform rod,} the shortest admissible time horizon is estimated {exactly}, and the {exact} optimal control law is symbolically found in a {recurrent way.} 
\end{abstract}


\section{Introduction}

Such classical mechanical systems as rods and strings as well as  related control problems have been attracting attention of mathematicians and engineers  for many years. This is not surprising taking into account how many physical processes are modeled by or simplified to the wave equation \cite{Krabs:1995}. As well as for other distributed parameter systems, a possible solution to a control problem for the wave equation over a finite horizon involves boundary and distributed control inputs \cite{Lions:1971, Butkovsky:1969}.

The boundary control seems to be more {feasible 
in mechanical systems}  since {its 
 realization} employs {actuators} 
 widely  used in engineering. However, it has certain limitations because a finite number of inputs is used to control the continuum system of partial differential equations (PDEs). {For 
  vibrating} systems {like} rods and strings such a limitation is a minimal control time, so that the system cannot be transferred to a {desired} 
   state  quicker than this critical
    time even by means of unconstrained control \cite{Butkovsky:1969, KostinGavrikovJCSSI:2021}. For more general vibrating systems, {e.g.} 
      with memory, it may even lead to  uncontrollability \cite{Romanov:2016}. 
{The distributed} control gives {some} 
 advantages since a continuum input provides means to {control} 
  each vibrating mode separately. Ideally, we are able to instantly transfer a system to a desired state \cite{Chen:1981}. To this end, the spectral theory of linear operators~\cite{Banks:1983, Curtain:1995}, the decomposition based on the Fourier method~\cite{Chernousko:1996}, the selection method~\cite{Gerdts:2008} as well as many other approaches are utilized. 

{ However, the exact controllability} 
{may be not possible for many real-world} implementations  of  dynamical systems  with distributed parameters \cite{Glowinski:1994} since a discrete numerical solution may develop {singularities.  
Usually, one has to discretize a distributed control input} first and only then {one} may apply it  to a mechanical system. This leads to a question: are such  finite-dimensional inputs derived from a continuum control law indeed optimal in the class of finite-dimensional control functions? {If one does not take this issue into account, one} may propose a spatial discretization of {an optimal  input} such that {a} rod (as well as more general vibrating systems) becomes uncontrollable over a fixed  time horizon \cite{Zuazua:2005}, especially if the control is applied only  along some part of the length \cite{Ho:1990}.  {Moreover, 
a} numerical solution to an optimal control problem (OCP) {might} be discretized in time also, and the same question of optimality as for spatial discretization  arises. 
 This issue also attracts a lot of attention \cite{Lagnese:2003, Heinkenschloss:2005, Liu:2020, Kroner:2011}.   We do not consider {such a discretization} in this paper assuming that at least polynomial in time 
 signals may be implemented with a good precision. 

In our study, we assume from the beginning that the control inputs are spatially finite-dimensional: boundary forces are applied at the {rod ends} and a piecewise {constant  force} is distributed along {the central line}. Such a force can be implemented with the {help of  piezoelectric} actuators or other control elements placed along {the entire  rod}.   Thus,  {our control} is discrete in space, although the piezoelectric stress  of each actuator itself is distributed uniformly on the corresponding subinterval and is varied in time.  
{For simplicity,} we do not consider a detailed model of the actuators and understand the applied forces as control inputs,   as well as we suppose  
that there are no {gaps} between control elements. Since piezoelectric actuators are widely used in applications \cite{Tzou:2019}, including in series utilization \cite{Mu:2019}, well-developed models, e.g. \cite{IEEEStandard:1988},  may be exploited to implement the proposed control approach as in \cite{Kumar:2008, Li:2017}. 

The assumption that the input is piecewise constant in space allows for splitting the 
 controlled system into 
  interconnected subsystems (cf.  \cite{Rice:2009, Massioni:2009}), each of which is described through   traveling waves  and is actuated by {one} 
control element. The continuity and boundary conditions   interweave algebraically these traveling waves. By using {a 
 mesh} on the time-space domain (cf. time decomposition in \cite{Lagnese:2003}), we 
 express all the conditions {as} 
a linear system, which solvability guarantees controllability of the dynamical system whereas unsolvability conditions provide  the critical time horizon.

In what follows, this splitting into subsystems is applied to a variational formulation of the original boundary value problem (BVP). The approach we use is based on the method of integro-differential relations (MIDR)~\cite{Kostin:2018}.  It has been developed for description of dynamics and control {of 
 elastic} systems {which  involves} the {Ritz and} FEM-type approximations \cite{Kostin:2018b, Kostin:2020a}. Additionally to displacements, a dynamic variable (so-called potential) is introduced, which {binds together} the momentum density and normal forces in the cross section.  The local constitutive relations  are {replaced with} a functional characterizing how well these relations  are satisfied in terms of kinematic and dynamic {variables. 
 This} constitutive functional is subject to boundary constraints  and continuity conditions on the interfaces between controlled subsystems and equals to zero on the exact solution expressed via traveling waves. The variational formulation gives certain advantages providing required smoothness of the solution and its traces without a priori assumptions as well as avoiding the employment of derivatives of delta-functions, {which} usually { 
  represent} piezoelectric forces, c.f.  \cite{Kucuk:2014}.

Next, we consider the OCP of minimizing  the mean mechanical energy stored by the rod during its motion. By utilizing the {d’Alembert’s} description, the control problem is reduced to a classical variational problem. The resulting Euler--Lagrange ordinary differential equations (ODEs) 
 together with appropriate boundary conditions constitutes a BVP, which solution provides the optimal control signals and the corresponding rod motion. Although our approach leads to the spatial discretization of the system, it differs from the standard discretization techniques such as finite element, volume, and difference methods~\cite{Balas:1986, Christofides:2001, Lewis:2004} since the solution to {this 
  BVP} 
   exactly represents the state of the original PDE system. 

{In}   \cite{Kostin:2022b}, we considered a simplified version of the OCP, 
{in which} the time horizon is not arbitrary but is a multiple of the length of the control element. In this case, a time-space mesh is not so dense and the number of auxiliary traveling waves is less since the characteristics propagating from initial and terminal vertices coincide.

In this paper, a new generalized formulation of the initial-boundary value problem (IBVP) is given. 
Contrary to conventional one-variable  statement (in displacements only), we  introduce a variational formulation in two state variables.
For spatially homogeneous rod, the OCP is reduced to a one-dimensional variational problem. 
As a result, we obtain an exact optimal solution. That makes it possible   {to estimate 
 (i)} the admissible time for bringing the system to an arbitrary terminal state, {(ii)} the minimum energy cost for a given transition, as well as  {(iii)} limiting properties of the motion if the number of control inputs increases. While the explicit solution can be found  for a uniform rod only, the proposed variational formulation allows one to develop efficient numerical procedures for solving a wider class of optimization problems by exploiting, for example,  the finite element method.

The paper is organized as follows. In Sect.~\ref{sec:2}, we introduce the controlled system, give the variational formulation of the IBVP, and  state the OCP. 
The special mesh on the time-space domain is defined and the solution algorithm for the direct dynamic problem is described in Sect.~\ref{sec:3}. In Sect.~\ref{sec:4}, the OCP 
 is solved by using  auxiliary wave functions, and  a numerical example is presented. Conclusions 
 are given in Sect.~\ref{sec:5}.

\section{Statement of the Control Problem}\label{sec:2}


Let us consider longitudinal motions of a thin rectilinear elastic rod. 
Its length in the undeformed state is $2L$ (see the scheme in Fig.~\ref{fig:01}). 
The  $x$-axis is directed along the central line with the origin at the middle of the rod. 
\begin{figure}[t]
	\begin{center}
\includegraphics[width=0.8\linewidth]{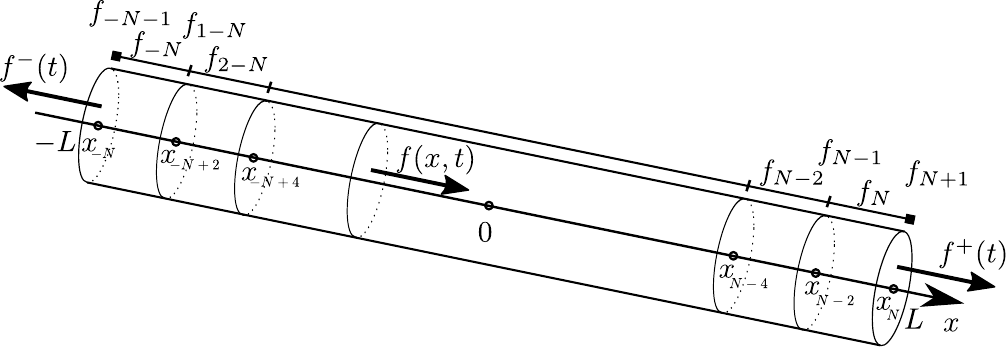}\\ 	
	\end{center}  
	\caption{Scheme of a rod with $N$ control elements.}
	\label{fig:01}
\end{figure}
The absolute displacements {of rod points} with the coordinate $x\in I_L:=(-L,L)$ at the time instant $t\in I_T:=(0,T)$ are given by a mapping $v:\Omega\rightarrow\mathbb{R}$, where $\Omega=I_T\times I_L$ is the time-space {domain. 
We} are also interested {in 
 the} linear momentum density $p:\Omega\rightarrow\mathbb{R}$ as well as the total force  $s:\Omega\rightarrow\mathbb{R}$ normal to the rod's cross {section.     
Mechanical} properties of the system are defined {by the  tension} stiffness $\kappa:I_L\rightarrow\mathbb{R}$ and the linear mass density $\rho:I_L\rightarrow\mathbb{R}$.
The rod is loaded {by external} normal forces $f^{\pm}:I_T\rightarrow\mathbb{R}$ applied at the ends with the coordinates $x=\pm L$.
Additionally to the elastic force $\kappa(x)v_x(t,x)$, an inner force $f:\Omega\rightarrow\mathbb{R}$, which is also normal to the cross section, stretches or contracts the rod along the $x$-axis. 
The subscripts $t$ and $x$ denote  the partial derivatives in time and space, respectively.

\subsection{Generalized statement of the IBVP}\label{sub:0202}

The key idea of the proposed approaches is that {the  state} variables of a physical state can always be  divided into two {groups: 
kinematic quantities (displacements, strains, velocities, temperature) and 
dynamic values (stresses, momenta, heat fluxes).}
 At the same time, governing equations can be split into three types: (i) initial and boundary conditions, (ii) balance and continuity laws, and (iii) constitutive relations. The constitutive relations connect {kinematic and dynamic} 
 variables and contain information on material properties of the studied system. In generalized statements (e.g. Hamilton, Reisner Hu--Wasidzu, Hellinger--Reissner principles), one usually assumes that some of the governing equations are weakened; these are typically balance equations~\cite{Washizu:1982}. The essence of the MIDR~\cite{Kostin:2018} is that equations of the third type are represented in the integral form, whereas the other equations must be considered as essential constraints. An IBVP which is  modified in accordance with this idea can be reduced to the minimization of a non-negative functional over all admissible state variables.

The variational formulation of the IBVP under study 
with respect to two variables, kinematic $v$ and dynamic $r$, is described as follows:
\begin{problem}\label{prob:21}
\textit{Given the {a.e.}  
  positive coefficients $\kappa,\rho\in L^{\infty}(I_L)$, the initial distributions $v_0:I_L\rightarrow\mathbb{R}$ and $r_0:I_L\rightarrow\mathbb{R}$ in the Sobolev space $H^{1}(I_L)$, the boundary force integrals $u^{\pm}:I_T\rightarrow\mathbb{R}$ with $u^{\pm}\in H^{1}(I_T)$ and $u^{\pm}(0)=0$, as well as the normal force $f:\Omega\rightarrow\mathbb{R}$ with $f\in L^{2}(\Omega)$,
find such functions $v^{*}(t,x)$ and $r^{*}(t,x)$ in $H^{1}(\Omega)$
that minimize  the constitutive functional}
	\begin{equation} \label{eq:207:constitutive_minimization}
		\begin{array}{c}\ds
			Q[v^{*},r^{*}] 
				= \min\nolimits_{v,r\in H^1(\Omega)}Q[v,r]=0,\quad
			Q = \int\nolimits_{\Omega}q\,\mathrm{d}\Omega\geq 0,
		\\
			q:=\frac14\big(g^2 + h^2\big),\quad\ds
		     g: = \sqrt{\rho}v_{t} - \frac{r_x}{\sqrt{\rho}}, \quad
		      h: = \sqrt{\kappa}v_{x} - \frac{r_t-f}{\sqrt{\kappa}},
		\end{array}
	\end{equation}
\textit{subject to the initial and boundary constraints}
      \begin{equation} \label{eq:202:initial_conditions}
      \begin{array}{c} 
			v(0,x) = v_{0}(x),\quad
			r(0,x) = r_{0}(x),\quad
			x \in I_L,\\
        	r(t,-L) = r_{0}(-L) + u^{-}(t),\quad
        	r(t,L) = r_{0}(L) + u^{+}(t),\quad
        	 t \in I_T.
      \end{array}        	 
	\end{equation}	
\end{problem}
Here, {the 
 functional} $Q$ reaches its absolute minimum on the {exact} solution. At that, the constitutive residual function { $q$ }
is equal to zero almost {everywhere in  $\Omega$. 
 The scaling} in~\eqref{eq:207:constitutive_minimization} is done so that {$Q$} 
has the dimension of {action. 
In numerics,}  
the nonzero value of $Q>0$ estimates a posteriori the integral error of an approximate solution, whereas the integrand $q$ {can  
 estimate} the local quality of the approximation.

Embedding the one-dimensional functions $r_0(x)$, $u^{\pm}(t)$ from the Hilbert space $H^1$ into the space of continuous functions ($H^1\subset C^0$) according to the Sobolev lemma 
\cite{Yoshida:1965} means that these functions can be continued respectively to the close interval $[-L,L]$, $[0,T]$
and their values must meet at points $x = \pm L$, $t = 0$ according to~\eqref{eq:202:initial_conditions}.

	\subsection{Control Forces, Force Jumps and Integrals}\label{sub:0203}

In Problem~\ref{prob:21}, the first weak {derivatives} of the functions $u^{\pm}$ define the boundary forces $f^{\pm}=(u^{\pm})'\in L^2(I_T)$.
In what follows, we rename these mapping as $u_{\pm N \pm 1}:=u^{\pm}$ so that $f_{\pm N \pm 1}:=f^{\pm}$ and consider {as lumped control inputs.}

The normal force $f$ is taken as a distributed control input.
We assume that  the function $f(t,x)$ is piecewise constant in space. 
No external linear force density, {e.g.}  a gravitational load, is applied along the $x$-axis. 
From a technical point of view, such a load {can be generated 
 by a set of piezoelectrical actuators attached on the rod's side surface.} 
On a given segment of the rod, the actuators should work symmetrically with respect to the $x$-axis to avoid bending {deformations. 
A} group of actuators together with the adjacent piece of the rod is named further a control element.
{These} $N$ elements have equal {lengths} and are  inseparably located along the central line: there {are no gaps} between adjacent control elements.
Moreover, electromechanical properties of all the elements are equivalent. 
It is also assumed {as for} 
 the simplest mathematical models, e.g. \cite{IEEEStandard:1988}, that the  {force  $f(t,x)$}  is constant along an element and {can be produced without constraints by}
the element's actuators. 
Functional restrictions on $f$ following 
 from Problem~\ref{prob:21} are discussed below.

{Applications} of 
systems that employ finite-dimensional distributed inputs usually involve elastic objects and piezoactuators/sensors. {These  systems } are sometimes referred as smart structures \cite{Chopra:2002}. They are also called phononic crystals {if they} 
 consist of identical sub-structures \cite{Li:2017}. Such {structures are} used for active
 and passive vibration damping \cite{Shengbing:2012, Lossouarn:2015},
 frequency filtration \cite{Degraeve:2014},  etc.
 {Although much more general (3D motion involving also friction, bending,
 and subject to finite deformation), close control problems arise in soft robotics 
 for peristaltic locomotion when an elongated elastic body made from periodic segments crawls 
 due longitudinal contraction/extension of segments caused by magnetic fields, pneumo- and servomotors, etc. \cite{Omori:2008, Guglielmino:2010, Seok:2013}.}

The control elements are naturally related to $N$ space {intervals:} 
	\begin{equation} \label{eq:203:rod_segments}
			x \in I^{x}_{k} := (x_{k-1},x_{k+1}),\quad
			k \in J_s;\quad
			x_{n} =\frac{n\lambda}{2},\quad
			n \in J_x,\quad 			
			x_{\pm N} = \pm L,
	\end{equation} 
where {$\lambda = 2L/N$} denotes the length of each element. 
The two sets of indices 
	\begin{equation} \label{eq:203:space_indexing_sets}
			J_{s} = \{ 1-N, 3-N,\ldots,N-1 \},\;\;
			J_{x} = \{ -N, 2-N,\ldots,N \}
	\end{equation} 
{in~\eqref{eq:203:rod_segments}  label} respectively the space intervals $I^{x}_{k}$ and the interface points $x_{n}$. 

{The  function} $f$ over each element does not depend on the space coordinate $x$, that is  {$f(t,x) = f_{k}(t),$ with $x \in I^{x}_{k},$ $k \in J_s,$ $t \in I_T.$ } 
Besides these {piezoelectric  forces,} the external loads $f_{-N-1}$ and $f_{N+1}$ complete a set of control functions $f_{k}\in L^2(I_T)$ with the indices  $k\in J_c$. Here, the supplemented index set 
		{$J_{c} = J_{s}\cup\{-N-1, N+1\}$} 
is related to the set of control inputs (see Fig.~\ref{fig:01}).


The variation of {the 
integral $Q$} in \eqref{eq:207:constitutive_minimization} subject to the boundary conditions \eqref{eq:202:initial_conditions} 
 shows that the jumps of  derivatives $[v_x(t,x_{n})]$ at {the 
 	points} $x_{n}$ and their boundary values  $v_x(t,x_{\pm N})$ depend exclusively on the differences of adjacent control functions 
      \begin{equation} \label{eq:204:jump_force}
       	 f_{n}:=f_{n+1}-f_{n-1},\quad
      	 n \in J_x.
      \end{equation}
The number of such functions is one less than the number of the original inputs $f_{k}$ {with $k \in J_c$. }
When the same control force $\bar{f}(t)$ acts  for each  $k \in J_c$ in the governing equations~{\eqref{eq:207:constitutive_minimization}--\eqref{eq:202:initial_conditions}}
 (in other words,  $f_{k}= f_{k+2}$ for $k \in J_c \setminus \{ N + 1 \}$), 
the {particular} solution $v(t,x)=0$, $r(t,x) = \int_{0}^{t} \bar{f}(\tau)d\tau$ appears, and
the rod  moves as if it is free of any loads. 
This means that {the 
	displacements} $v$ do not depend on the sum of the control signals $f_{k}(t)$. 
This value $\bar{f}(t)$ affects only the intensity of residual {stresses 
	$s$,} which do not influence the rod's mean energy minimized in the OCP considered in {Subsec.~\ref{sub:0205}. }
For definiteness, we reduce these stresses by zeroing the sum:
      \begin{equation} \label{eq:204:zero_force_sum}
       	 \frac{1}{N+2}\sum_{k \in J_c} f_{k}(t) = \bar{f}(t) = 0, \quad
       	 {t\in I_T}.
      \end{equation}

{For 
 convenience,} let us define two vector spaces with elements $\boldsymbol{f}_{c}: {I_T} \rightarrow \mathbb{R}^{N+2}$  and $\boldsymbol{f}_{x}: {I_T} \rightarrow \mathbb{R}^{N+1}$ that are respectively  a $(N + 2)$-tuple and  a $(N + 1)$-tuple of time-dependent functions according to 
      \begin{equation} \label{eq:204:force_vectors}
{\boldsymbol{f}_c := (f_{-N-1}, f_{-N+1},\ldots,  f_{N+1}),\quad
			\boldsymbol{f}_x := (f_{-N}, f_{-N+2},\ldots,  f_{N}).}
      \end{equation}
Here,  $\boldsymbol{f}_{c}$ contains both boundary and distributed control inputs, whereas $\boldsymbol{f}_{x}$ groups the control jumps~\eqref{eq:204:jump_force}.   
Given the control vector $\boldsymbol{f}_{x}$, {the linear system 
~\eqref{eq:204:jump_force}, \eqref{eq:204:zero_force_sum}} can be resolved with respect to the entries of $\boldsymbol{f}_{c}$. As a result, the control vector-valued function $\boldsymbol{f}_c(t)$ 
 is expressed through
the vector of the control jumps $\boldsymbol{f}_x(t)$ by
      \begin{equation} \label{eq:204:control_relations}
{\boldsymbol{f}_{c}(t)  =\boldsymbol{F}^{-1}{\boldsymbol{\hat f}}_{x}(t),\quad      
{\boldsymbol{\hat f}}_{x} = (\boldsymbol{f}_x, 0),\quad      \boldsymbol{F}\in \mathbb{R}^{(N+2) \times (N+2)}.}
      \end{equation}

The force integrals are also introduced according to
      \begin{equation} \label{eq:206:control_integrals}
          u_{k}(t)=\int_{0}^{t} f_{k}(\tau)\,\mathrm{d}\tau, \quad
          k \in J_c,\quad
          t \in {I_T}.
      \end{equation}
By taking into account~\eqref{eq:204:force_vectors}, the linear algebraic constraint 
	{$\sum_{k \in J_c} u_{k}(t) = 0,$ $t \in I_T,$}
is imposed on these functions. 
Similarly to~\eqref{eq:204:jump_force} and in agreement with~\eqref{eq:206:control_integrals}, we can also define the jumps of control integrals 
       	 {$u_{n}(t)=u_{n+1}(t)-u_{n-1}(t),$ $n \in J_x.$}
 In accordance with~\eqref{eq:204:force_vectors}, two vector-valued functions $\boldsymbol{u}_{c}:{I_T}\rightarrow \mathbb{R}^{N+2}$ and $\boldsymbol{u}:{I_T}\rightarrow \mathbb{R}^{N+1}$ {are introduced such that $\boldsymbol{u}'_c = \boldsymbol{f}_c,$ $\boldsymbol{u}' = \boldsymbol{f}_x$, $\boldsymbol{u}_c=(u_k)_{k\in\boldsymbol{J}_c}$, $\boldsymbol{u}=(u_n)_{n\in\boldsymbol{J}_x}$.}  
 {We consider in the sequel the entries of $\boldsymbol{u}$ as control inputs for Problem~\ref{prob:21}.} 

		\subsection{Relation to the classical wave equation}\label{sub:0205}

 Problem~\ref{prob:21} as a generalized formulation must {admit} a classical solution. 
To show that, let us suppose that $\kappa$ and $\rho$ are continuous functions.
The {first PDE} governing the rod's {motion 
links} the momentum $p$ and the force $s$ according to Newton's second law {as follows} 
	\begin{equation} \label{eq:204:Newton_law}
		p_{t}(t,x) = s_{x}(t,x).
	\end{equation}
We define the  dynamic potential  $r$ such that
    \begin{equation} \label{eq:205:potential_relations}
        p(t,x) = r_x(t,x), \quad
        s(t,x) = r_t(t,x).
    \end{equation}
This representation {of 
 $p$ and 
  $s$} satisfies automatically the balance equation~\eqref{eq:204:Newton_law} if the second derivatives of 
   $r$ exist. 
The equality of the functional $Q$ to zero in the case of piecewise smooth functions $v$ and $r$ with accounting for~\eqref{eq:205:potential_relations} means that
	\begin{equation*} 
	\begin{array}{c}\ds
        	g = \sqrt{\rho}v_{t} - \frac{r_x}{\sqrt{\rho}}=\sqrt{\rho}v_{t} - \frac{p}{\sqrt{\rho}} =0, \quad
        	h = \sqrt{\kappa}v_{x} - \frac{r_t-f}{\sqrt{\kappa}}=\sqrt{\kappa}v_{x} - \frac{s-f}{\sqrt{\kappa}} = 0.
	\end{array}
	\end{equation*}
This leads to two local constitutive relations
      \begin{equation} \label{eq:204:constitutive_law}
         	p(t,x) = \rho(x)v_t(t,x), \quad 
         	s(t,x) = \kappa(x)v_x(t,x) + f(t,x)
      \end{equation}
between the momentum $p$ and the velocity $v_t$ as well as
between the forces $s$ and the longitudinal strains $v_x$ (Hooke's law).
Substituting the expressions for $p$ and $s$ from~\eqref{eq:204:constitutive_law}  in~\eqref{eq:204:Newton_law} and taking into account that $f$ is a piecewise constant function of $x$, we recover the wave equation
		{$v_{tt}(t,x) = v_{xx}(t,x),$ 
		$(t,x) \in \Omega.$}

After differentiating the second equation in~\eqref{eq:202:initial_conditions} w.r.t $x$ and extracting the velocity $v_t$ of the rod's points from~\eqref{eq:204:constitutive_law} and~\eqref{eq:205:potential_relations},  initial conditions are imposed
on both the displacements $v$ and the velocity $v_t$ by
      \begin{equation} \label{eq:204:initial_distributions}
          v(0,x) = v_{0}(x),\quad
          v_t(0,x) = r'_{0}(x)/\rho(x).
      \end{equation}
Additionally,  inhomogeneous  boundary conditions of second kind are defined by differentiating~\eqref{eq:202:initial_conditions} w.r.t. $t$ and accounting for~\eqref{eq:204:constitutive_law} as follows
      \begin{equation} \label{eq:204:boundary_distributions}
        	\kappa(\pm L)v_x(t, \pm L) = f_{\pm N \pm 1}(t) - f_{\pm N \mp 1}(t).
      \end{equation}
      
Finally,  continuity conditions for displacements $v$ and forces $s$ 
	\begin{equation} \label{eq:204:interelement_distributions}
	\begin{array}{c}\ds
       	 [v(t,x_{n})] = 0, \quad
      	 n\in J_{x} \setminus \{ -N, N \};\quad
     	      [\kappa(x_{n})v_x(t,x_{n})] =  f_{n+1}(t) - f_{n-1}(t)
	\end{array}
	\end{equation}
must be {imposed 
at 
points $x_n$ with $[F (x_n)]:=F(x_n+0)- F(x_n-0)$.}  
The first equation in~\eqref{eq:204:interelement_distributions} arises due to rod integrity, while the second one follows from Newton's third law and defines the interface force balance.  
For shortness,  the relations~\eqref{eq:204:initial_distributions}--\eqref{eq:204:interelement_distributions} are called  the interface conditions.

		\subsection{Optimal Control Problem}\label{sub:208}

The following OCP  
is considered.
\begin{problem}\label{prob:22}
\textit{Find the control vector-valued function $\boldsymbol{u}^{*}\in H^1(I_T;\mathbb{R}^{N+1})$ and the terminal constant $c^*_1$ such that the mean mechanical energy $E$  stored in the rod over the fixed time interval $I_T$ reaches its minimum}   
	\begin{equation*} 
		\begin{array}{c}\ds
			E[v,r,\boldsymbol{u}]\to \min\nolimits_{\boldsymbol{u},c_1},\quad
      		E = \frac{1}{T}\int\nolimits_{\Omega}e \,\mathrm{d}\Omega, \quad 
        		{e:}=\frac{\rho v_{t}^{2}}{4}+\frac{\kappa v_{x}^{2}}{4}
        			+\frac{(r_{t}-f)^{2}}{4\kappa}+\frac{r_{x}^{2}}{4\rho} ,
		\end{array}
	\end{equation*}
\textit{subject to the integral equality $Q[v,r,\boldsymbol{u}]=0$ from~\eqref{eq:207:constitutive_minimization},
the initial and boundary conditions~\eqref{eq:202:initial_conditions} 
 (see Problem~\ref{prob:21}), 
as well as the terminal conditions}
    \begin{equation*} 
		\begin{array}{c}\ds
        		v(T,x)=v_{1}(x),\quad
        		r(T,x)=r_{1}(x)
			=\int\nolimits_{-L}^{x}p_{1}(\xi)\,\mathrm{d}\xi+c_{1},
			\\ \ds
        		v_{1}\in H^1(I_L),\quad
        		p_{1}\in L^2(I_L),\quad
        		x\in  \bar{I}_L.
		\end{array}
	\end{equation*}
\end{problem}	
{The energy} $E$ depends on the control variable $\boldsymbol{u}$ through the control function $f$ as well as the problem constraints. 
The mapping $e:\Omega \rightarrow \mathbb{R}$ denotes the linear energy density.
The desired functions of displacements $v_{1}:\bar{I}_L\rightarrow\Omega$ and momentum density $p_{1}:\bar{I}_L\rightarrow\Omega$ completely define the terminal state of the elastic rod, whereas the parameter $c_{1}$ 
does not influence this state.

In {Sect.~\ref{sec:3}}, {the} 
 exact solution of the direct dynamic {problem 
is found assuming} that the control function {$\boldsymbol{u}(t)$} is given and the rod is homogeneous. In {Sect.~\ref{sec:4}}, we present a solution algorithm 
reducing the {OCP} 
 in two-dimensional time-space domain (Problem~\ref{prob:22}) to {a} one-dimensional variational problem. {Whereas} 
  usually a solution to an OPC for PDEs in general and {for the} wave equation in particular can be obtained only approximately, for example, by means of Fourier  \cite{Butkovsky:1969} or finite difference \cite{Ervedoza:2013}  methods, we present a way to explicitly derive an analytical solution. Since we focus on a continuous system with  finite-dimensional control inputs while solving the OCP  rigorously, such a solution may serve as a benchmark in both theoretical and engineering studies employing distributed loads.

\section{{Solution to the Direct Dynamic Problem for a Uniform Rod}}\label{sec:3}

In what follows, we constrain ourselves to a particular case of a uniform elastic rod, which mechanical parameters $\rho (x) = \mathrm{const}$ and $\kappa(x) = \mathrm{const}$ do not depend on the spatial coordinates. 
{For simplicity,} 
we introduce dimensionless variables according to	{$v= v^{*},$ $r=\kappa \tau_{*}r^{*},$ $x=L x^{*},$ $t=\tau_{*} t^{*},$ $\tau^2_{*}=L^2\rho/\kappa.$} 
The star superscript is further omitted. After this transformation, the length of the rod is equal to 2, $I_L=[-1,1],$ whereas the length of each element is $\lambda = 2/N$.
{Problem~\ref{prob:22} is  reformulated as follows}
\begin{problem}\label{prob:31}
\textit{Find the control function $\boldsymbol{u}^{*}\in H^1(I_T;\mathbb{R}^{N+1})$ and the  constant $c^*_1$ such that}
	\begin{equation}\label{eq:301:homogeneous_control_problem} 
			E[v,r,\boldsymbol{u}]\to\min\nolimits_{\boldsymbol{u},c_{1}},\quad v,r \in H^1(\Omega),
      \end{equation}
\textit{subject to the following constraints} 
	\begin{equation}\label{eq:301:boundary_value_problem} 
       	\begin{array}{c}\ds
			Q[v,r,\boldsymbol{u}]=0;\quad
			v(0,x) = v_{0}(x),\quad
			r(0,x) = r_{0}(x),
		\\ \ds
        	v(T,x)=v_{1}(x), \quad
        	r(T,x)=r_{1}(t), \quad
			x \in  \bar{I}_L;
		\\ \ds
			[v(t,x_{n})] = [r(t,x_{n})] = 0,	\quad
			n \in J_{x} \setminus \{ -N, N \}, 
		\\ \ds
			r(t,\pm 1) = r_{0}(\pm 1) + u_{\pm N \pm 1}(t),\quad
			t \in {I_T}.
       	\end{array}
      \end{equation}
\textit{Here, the functionals $E$ and $Q$ take the form}
	\begin{equation}\label{eq:301:homogeneous_ functionals} 
       	\begin{array}{c}\ds
      		E = \frac{1}{T}\int\nolimits_{\Omega}e \,\mathrm{d}\Omega, \quad
        		e=\frac{v_{t}^{2}}{4}+\frac{v_{x}^{2}}{4}
				+\frac{(r_{t}-f)^{2}}{4}+\frac{r_{x}^{2}}{4}, 
		\\[2ex]
		\ds
			Q = \int\nolimits_{\Omega}q\,\mathrm{d}\Omega, \quad
		      q=e - \frac{v_{t} r_{x}}{2} - \frac{v_{x}(r_{t}-f)}{2}. 
       	\end{array}
      \end{equation}
\end{problem}

		\subsection{Representation of the Solution in d'Alembert's Form}\label{sub:302}
		
To analyze {the 
	dynamics} of a uniform rod described by~\eqref{eq:301:homogeneous_control_problem} and~\eqref{eq:301:boundary_value_problem}, a traveling wave representation {of 
	$v, r$} in d'Alembert's form is applied. 
We assume that on each subdomain $\Omega_{k}={I_T}\times I^{x}_{k} \subset \Omega$ with $k \in J_{s}$, {the unknown 
	variables} are represented as
	\begin{equation} \label{eq:302:dAlembert_solutions}
		\begin{array}{l}\ds
			v(t,x) = w^{+}_{k}(t+x) + w^{-}_{k}(t-x),\quad
			r(t,x) = w^{+}_{k}(t+x) - w^{-}_{k}(t-x)  + u_{k}(t),
		\end{array}
	\end{equation}  
where {left ($w^+_k$) and right ($w^-_k$)} traveling waves
{$w^{\pm}_{k}: I^{\pm}_{k} \rightarrow \mathbb{R},$ $w^{\pm}_{k}\in C(I^{\pm}_{k}),$} 
are introduced with the domains
	\begin{equation} \label{eq:302:dAlembert_domains} 
			I^{\pm}_{k} = \left( z^{\pm}_{k}, z^{\pm}_{k} + \lambda + T \right),\quad
			z^{+}_{k} = \frac{k-1}{2}\lambda,\quad
			z^{-}_{k} = -\frac{k+1}{2}\lambda.
	\end{equation}  
Each interval $I^{\pm}_{k}$ for $k \in J_s$ is defined through the infimum and the supremum of the arguments $t\pm x$ over $\Omega_{k}$. 
Due to the symmetry of the rod with respect to the origin point $x=0$, these intervals relate  as $I^{\pm}_{k}=I^{\mp}_{-k}$. 	
Further, we show that representation {\eqref{eq:302:dAlembert_solutions} 
 is} valid by explicitly resolving equations arising due to initial, terminal, boundary and interelement conditions defined in \eqref{eq:301:boundary_value_problem}.		

		
A rather convenient  coordinate representation of the traveling waves {$w^{\pm}_{k}$ 
 with} $k\in J_s$  is given in the coordinate frame $(z^{+}, z^{-})$ rotated counter-clockwise on the angle $\frac{\pi}{4}$ with respect to the frame $(t, x)$.
	The direct and inverse transformations of the new and old coordinates have the form
	\begin{equation} \label{eq:303:coordinate_transformation} 
		\begin{array}{c}
			z^{+} = t + x,\quad
			z^{-} = t - x;\quad
			t = \frac12 z^{+} + \frac12 z^{-},\quad
			x =\frac12 z^{+} - \frac12 z^{-}.
		\end{array}
	\end{equation}  
	In Fig.~\ref{fig:02}, the $z^{\pm}$-axes are presented by blue solid lines, the orts of the new frame
$\boldsymbol{i}^{\pm}$ with the coordinates   $ \langle \frac{1}{2},\pm\frac{1}{2}\rangle$ in the old coordinate system are depicted  by two blue arrows. 
All characteristic lines in the domain $\Omega$ can be set by the equations $z^{\pm} = \mathrm{const}$. {In the new coordinates \eqref{eq:303:coordinate_transformation} the traveling waves $w^{\pm}_{k}$ with $k \in J_s$ depend respectively on the only argument $z^{\pm}$.} 
	
		\subsection{Mesh on the Time-Space Domain}\label{sub:304}
		
{Although the variables $v,r$ of the OCP~\eqref{eq:301:homogeneous_control_problem}, \eqref{eq:301:boundary_value_problem}, which are represented 
by~\eqref{eq:302:dAlembert_solutions}, 
 satisfy the integral constraint $Q=0$,} the solution has been given so far only for the union of disjoint open subdomains $\bigcup_{k \in J_{s}} \Omega_{k}$, but not on its  closure $\overline{\Omega} = \bigcup_{k \in J_{s}} \overline{\Omega}_{k}$. 
By taking into account the initial and terminal constraints (at $t=0, T$) in~\eqref{eq:301:boundary_value_problem} as well as boundary and interelement  constraints (at $x=x_{n}$ with $n\in J_x$), the solution $(v,r)$ has to be extended to the set of interface edges $\overline{\Omega}  \setminus \bigcup_{k \in J_s} \Omega_{k}$ of measure zero.  
The parts of this set are presented in Fig.~\ref{fig:02} for $N=4$ with thick {vertical and horizontal} lines.


		\begin{figure}[t]
			\begin{center}
\includegraphics[width=0.75\linewidth]{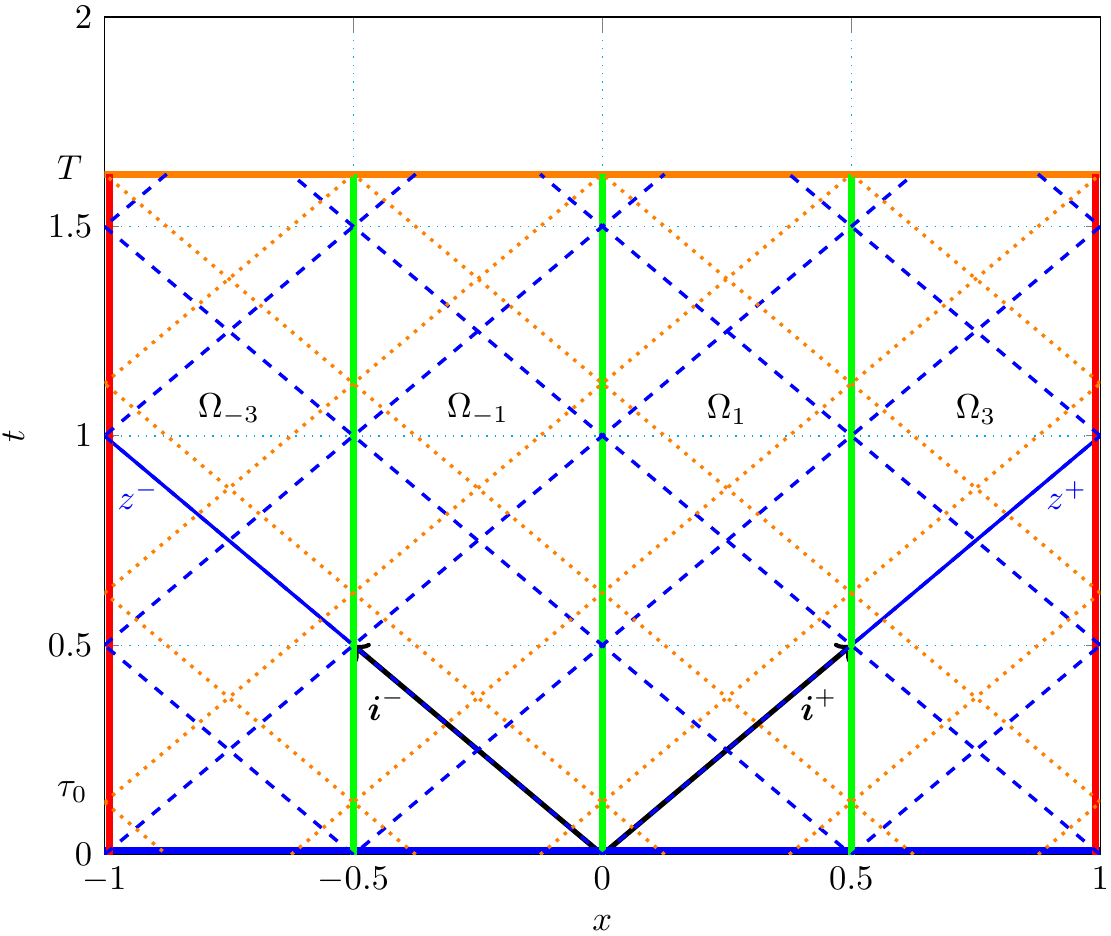}\\ 				
				\caption{Non-uniform double mesh on $\Omega=(0,1\frac58)\times(-1,1)$ with $N=4$, $M=3$, $\tau_{0}=\frac18$.}
				\label{fig:02}
			\end{center}
		\end{figure}

To resolve 
 the above-mentioned continuity conditions, let us consider first {the 
 time} horizon $T$ which is not divisible by the segment length $\lambda$: 
 $T=M\lambda+\tau_{0}$ with $M\in \mathbb{N}$ and $0<\tau_{0}<\lambda$. 
Let us also introduce the complementary constant $\tau_{1}=\lambda-\tau_{0}$. 
The special case  $\tau_{0}=0$ will be discussed later. 
{If} $\tau_{0} \ne 0$, we put a mesh on the time-space rectangle $\overline{\Omega}$. The mesh is generated by the interface line segments defined in the coordinates $(t, x)$ as $S^{\alpha} = \left\{(t,x)\in \overline{\Omega}: t=\alpha \right\},$ $\alpha=0,T,$ $S^{x}_{n} = \left\{(t,x)\in \overline{\Omega}: x=x_{n} \right\},$ $n\in J_x,$
 together with the characteristic line segments in the coordinates $(z^{+}, z^{-})$   $S^{\pm}_{2j+i} = \{(z^{+},z^{-})\in \overline{\Omega}: z^{\pm}=\frac{2j-N}{2}\lambda + i\tau_{0}\},$ $i=0,1,$ $j \in \mathbb{Z}.$
Since 
 both coordinates $z^{\pm} \in [-1, T+1]$ inside $\overline{\Omega}$, the  subscripts {$i\in J_z$} of non-empty {segments} $S^{\pm}_{i}$ are explicitly defined by the {index} set\linebreak $J_{z} = \left\{0,1,\ldots, 2M+2N+1 \right\}.$ 
In Fig.~\ref{fig:02}, the characteristic segments of the mesh with $M=3$, $N=4$, and $\tau_{0}=\frac18$ are depicted { by slanting 
	lines.}
At that, $\lambda=\frac12$ and { $T=1\frac58$. }
Similarly to the space intervals $I^{x}_{k}$ and points $x_{n}$ 
 in~\eqref{eq:203:rod_segments}, the duration intervals $I^{t}_{l}$ and the time instants $t_{m}$ generated by the mesh 
are given by 
	\begin{equation} \label{eq:305:time_intervals}
{		\begin{array}{c}\ds
			t \in I^{t}_{l} := (t_{l},t_{l+1}),\
			l \in J_{d}=J_{t}\setminus \{2M+1\},\
				J_{t} := \left\{ 0, 1,\ldots, 2M+1 \right\},
		\\ \ds
			t_{m} = j\lambda+i\tau_{0}\quad 
			(t_{2M+1} = T),\quad 
			i=0,1, \quad 
			2j+i=m,\quad m \in J_{t}.
		\end{array}		
}		
	\end{equation} 

		\subsection{Double Indexing of Traveling Waves and Control Functions}\label{sub:306}
		
Let us discuss how the solution $(v,r)$ of the direct dynamic problem~\eqref{eq:301:boundary_value_problem} can be extended over each closed subdomain $\overline\Omega_{k}$, where {$\Omega_{k} = (0,T)\times I^{x}_{k}$, $k\in J_s$} (see Fig.~\ref{fig:02} as an {example).
We} denote the edges parallel to the $x$-axis as 
	\begin{equation} \label{eq:306:time_edges}
{S^{\alpha}_{k} := \left\{(t,x)\in\overline\Omega: t=\alpha,\; x\in I^{x}_{k} \right\},	\quad \alpha=0,T, \quad 	    k \in J_s,}
	\end{equation} 
where $J_s$ is the set of segment indices introduced in~\eqref{eq:203:space_indexing_sets}.
The edges parallel to the $t$-axis are represented as follows 	
	\begin{equation} \label{eq:306:space_edges}
		\begin{array}{c}
			S^{x}_{n,m} := \left\{ (t,x)\in\overline\Omega: t\in I^{t}_{m},\; x = x_{n} \right\},\quad
			m\in J_{d},\quad
			n\in J_{x},
		\end{array}		
	\end{equation} 		
where the intervals  $I^{t}_{m}$ are given in~\eqref{eq:305:time_intervals}. 		
To resolve the interface constraints in~\eqref{eq:301:boundary_value_problem}, we need to satisfy the continuity conditions over each  edge $S^{0}_{k}$ and $S^{T}_{k}$  from~\eqref{eq:306:time_edges} as well as over each 
 edge $S^{x}_{n,m}$ 
 \eqref{eq:306:space_edges}.
To operate with the values of $w^{\pm}_{k}$ on each of these edges in $\overline{\Omega}_{k}$, we divide their domains $I^{\pm}_{k}$ introduced in~\eqref{eq:302:dAlembert_domains} into the open intervals  
	\begin{equation} \label{eq:306:wave_function_intervals}
			I^{\pm}_{k,m} = \left( z^{\pm}_{k,m},  z^{\pm}_{k,m+1} \right),\quad
			z^{\pm}_{k,m} = z^{\pm}_{k} + j\lambda+i\tau_{0},\quad
			i=0,1,
	\end{equation} 		
{where $m =2j+i\in J_{w}=J_{t}\cup\{2M+2\},$ the set} of indices {$J_{t}$} is introduced according to~\eqref{eq:305:time_intervals}, and the characteristic coordinate $z^{\pm}_{k}$ is given in~\eqref{eq:302:dAlembert_domains}.	

The new edge traveling waves $w^{\pm}_{k,m}$ are defined so that
	\begin{equation} \label{eq:306:edge_wave_functions}
			w^{\pm}_{k,m}(z) = w^{\pm}_{k}\left( z+z^{\pm}_{k,m} \right), \quad
			z\in (0,\tau_{i}),\quad
			i=0,1,
	\end{equation} 		
with $k \in J_{s},$ $m = 2j+i \in J_{w}.$  Here, the coordinate shifts $z^{\pm}_{k,m}$ are expressed in~\eqref{eq:306:wave_function_intervals}, {the  sets} $J_{s}$ and $J_{w}$ are introduced in~\eqref{eq:203:space_indexing_sets} and {after}~\eqref{eq:306:wave_function_intervals}, respectively. 

The similar procedure is also applied to all the control functions $u_{n}(t)$ given on $t\in[0,T]$. 
Each of them is split into $m$ edge maps {$u_{n,m}:$} 
	\begin{equation} \label{eq:306:edge_control_functions}
			u_{n,m}(z) = u_{n}\left( z + t_{m} \right), \quad
			z\in (0,\tau_{i}), \quad
			i=0,1,\quad
	\end{equation} 	
where {$m =2j+i \in J_{d},$ $n\in J_{x}\cup J_{c},$ and} the time instants $t_{m}$ are given in~\eqref{eq:305:time_intervals}.	


		\subsection{Continuity Conditions for the State Variables}\label{sub:307}

We consider first the edges $S^{0}_{k}$ of the segment $S^{0}$ defined  in~\eqref{eq:306:time_edges}.
The two initial conditions from~\eqref{eq:301:boundary_value_problem} expressed in d'Alembert's form~\eqref{eq:302:dAlembert_solutions} are resolved on { $S^{0}_{k}$ as}
	\begin{equation} \label{eq:307:edge_initial_conditions} 
			{w^{\pm}_{k,i}(z) = {\textstyle\frac{1}{2}} v_{0}(\pm z^{\pm}_{k}\pm  i\tau_{0}\pm z) \pm {\textstyle\frac{1}{2}} r_{0}(\pm z^{\pm}_{k} \pm i\tau_{0} \pm z), }
	\end{equation}  
{where 	$z\in (0,\tau_{i}),$ $i=0,1$ $k\in J_s$.} Similarly, the terminal conditions on the edge $S^{T}_{k}$ of the segment $S^{T}$ are resolved as	
	\begin{equation} \label{eq:307:edge_terminal_conditions} 
{w^{\pm}_{k,2M+1+i}(z) = 
				{\textstyle\frac{1}{2}} v_{1}(\pm z^{+}_{k}\pm i\tau_{1}\pm z) \pm {\textstyle\frac{1}{2}} r_{1}(\pm z^{\pm}_{k}\pm i\tau_{1} \pm z)}	
	\end{equation}  
{with $z\in (0,\tau_{1-i}),$ $i=0,1,$ $k\in J_s.$} 	
There are totally $8N$ relations associated with the initial and terminal edges in accordance with~\eqref{eq:307:edge_initial_conditions} and~\eqref{eq:307:edge_terminal_conditions}. 
	
The boundary conditions on the edges $S^{x}_{\pm N, m}$ are expressed as
	\begin{equation} \label{eq:307:edge_boundary_conditions} 
{w^{\pm}_{\pm 1 \mp N,m}(z) \mp w^{\mp}_{\pm 1 \mp N, m+2}(z) = \mp u_{\mp N, m}(z)+r_{0}(\mp 1), }	
	\end{equation}  
{where $z\in (0,\tau_{i}),$ $i=m\,\mathrm{mod}\, 2,$ $m\in J_{d}.$} The number of the edge boundary relations  is equal to $4M+2$.
The continuity conditions on the inner segments $S^{x}_{n, m}$ can be represented according to 
	\begin{equation} \label{eq:307:edge_interelement_conditions} 
		\left\{
		\begin{array}{l}\ds
			w^{+}_{n-1,m+2}(z) + w^{-}_{n-1,m}(z) -
			w^{+}_{n+1,m}(z) - w^{-}_{n+1,m+2}(z) = 0 
		\\ \ds
			w^{+}_{n-1,m+2}(z) - w^{-}_{n-1,m}(z) -
			w^{+}_{n+1,m}(z) + w^{-}_{n+1,m+2}(z)  = u_{n,m}(z)
		\end{array}\right.
		\\
	\end{equation}  
{with $z\in (0,\tau_{i})$ and $i=m\,\mathrm{mod}\, 2,$ $m \in J_{d},$ $n\in J_{x} \setminus \{-N,N\}.$}  There are $(4M+2)(N-1)$ equations related to these edges.

Altogether, the number of the edge constraints equals to $N_{e} = 4MN + 10N$. 
The system~\eqref{eq:307:edge_initial_conditions}--\eqref{eq:307:edge_interelement_conditions} 	
contains  $N_{v} = N_{w} + N_{u} = 6MN + 2M + 7N+1$ unknowns, 
where $N_{w} = 2(2M+3)N$ is the number of traveling wave functions $w^{\pm}_{k,m}$ and $N_{u} = (2M+1)(N+1)$ is the number of control jump functions $u_{n,m}$.
As a result, the number of variables for the double mesh exceeds the number of equations if  the value $MN$ is rather large.

		\subsection{{Solvability of the System of Constraints}}\label{sub:308}

The control { of an elastic rod with one piezoelement ($N=1$) is equivalent to the control of the rod exclusively by the external boundary forces $f^{\pm}$, and was described in~\cite{KostinGavrikovJCSSI:2021}}. 
Thus, only the  case $N>1$ is studied further. 
{Then for $M>1$ the following theorem holds.}

\begin{theorem}
 For any initial and terminal functions $v_{0},r_{0},v_{1},r_{1}\in H^{1}(I_{L})$ and a fixed time {horizon}  $T=\lambda M +\tau_0,$ $M>1,$ $\tau_0\neq 0$, {the set of solutions} $v,r \in H^{1}(\Omega)$ to the BVP~\eqref{eq:301:boundary_value_problem} {is nonempty}. The solutions are expressed algebraically in terms of the traveling waves $w^{\pm}_{k,m}$ and control maps $u^{\pm}_{n,m}$ introduced in~\eqref{eq:306:edge_wave_functions} and~\eqref{eq:306:edge_control_functions}.
\end{theorem}

\begin{proof} Since the BVP \eqref{eq:301:boundary_value_problem} is equivalent to the linear system \eqref{eq:307:edge_initial_conditions}--\eqref{eq:307:edge_interelement_conditions} as shown in Subsect.~\ref{sub:302}--\ref{sub:307}, we prove the  statement of this theorem by providing an explicit algorithm solving \eqref{eq:307:edge_initial_conditions}--\eqref{eq:307:edge_interelement_conditions}. 

For $T=M\lambda+\tau_{0}$ with $\tau_{0}\ne 0$, the variable surplus in \eqref{eq:307:edge_initial_conditions}--\eqref{eq:307:edge_interelement_conditions} is equal to 
\begin{equation} \label{eq:308:nonuniform_variable_surplus} 
	N_{s}:=N_{v} - N_{e},\quad 
	N_{s}(M,N)= 2MN +2M - 3N +1.
\end{equation}  
The function $N_{s}$ monotonically increases with $M$. Thus, to prove that the system~\eqref{eq:307:edge_initial_conditions}--\eqref{eq:307:edge_interelement_conditions} is underdetermined for $M \ge 2$, consider first the case $M = 2$. Then the surplus is equal to $N_{s}(2,N)= N+5>0$. 
For each increase in $M$ by one, {$2N+2$ unknowns are added} towards $N_{s}$. 

We resolve  the underdetermined nonhomogeneous linear system~\eqref{eq:307:edge_initial_conditions}--\eqref{eq:307:edge_interelement_conditions} in the following way.
The algorithm contains four steps if the element numbers $N$ is odd and five if $N$ is even. The first four steps are common for any $N$.

\textit{Step I.} The traveling waves $w^{\pm}_{k,m}$ for $m=0,1$ and $m=2M+1,2M+2$ with $k\in J_{s}$ are expressed respectively  { through 
\eqref{eq:307:edge_initial_conditions}, \eqref{eq:307:edge_terminal_conditions}.}

\textit{Step II.} The control functions $u_{\pm N,m}$ for $m\in {J}_{t}\setminus \{2M+1\}$ are found by using the boundary conditions~\eqref{eq:307:edge_boundary_conditions}.  

\textit{Step III.} The interelement conditions from~\eqref{eq:307:edge_interelement_conditions} for $n\in J_{x} \setminus \{-N,0,N\}$ and $m=0,1,2M-1,2M$ are resolved. 
The pairs of expressed variables are chosen depending on the indices $n$, $m$. 
The first element of these pair is  the control function $u_{n,m}$. 
The second variable is chosen as {$w^{\pm}_{n\mp 1,m+2}$ if $\pm n<0$ and $m=0,1$ or $w^{\pm }_{n\pm 1,m}$ if $\pm n>0$ and $m=2M-1,2M$.}
	
\textit{Step IV.} The `inner' interelement conditions~\eqref{eq:307:edge_interelement_conditions} with the  indices $n \ne 0$ and $m =3,4,\ldots,2M-2$ are resolved with respect to two traveling waves 
{$w^{\pm }_{n\mp 1,m+2}$ and $w^{\mp}_{n\mp 1,m}$ for $\pm n<0$.}
Similarly  { to Step III} 
these functions are defined on   { either 
	$\Omega_{n-1}$} or $\Omega_{n+1}$ depending on which is closer to the boundary segments $S^{x}_{\pm N}$.  

{\textit{Step V.} The} conditions of 
continuity~\eqref{eq:307:edge_interelement_conditions} for $n=0$ {are satisfied} if $N$ is even. 
The pairs of equations for $m=0,1$ are resolved with respect to the functions $w^{+}_{-1,m+2}$ and $w^{-}_{1,m+2}$.
The other $2M-1$ pairs with $m>1$ are satisfied by $w^{-}_{-1,m}$ and $w^{+}_{1,m}$.

Although not unique, the proposed scheme applied to the system~\eqref{eq:307:edge_initial_conditions}--\eqref{eq:307:edge_interelement_conditions} confirms its solvability for $M>1$. $\qed$
\end{proof}

Note that  the solution $(v,r)$ to  \eqref{eq:301:boundary_value_problem} is continuous by construction for continuous initial conditions and free variables.
{The rank of the coefficient matrix in the system~\eqref{eq:307:edge_initial_conditions}--\eqref{eq:307:edge_interelement_conditions} is equal to the total number of equations.} In its turn, there can be more than one square submatrices of the coefficient matrix with this rank. In this sense, the choice of free variables may turn out to be not unique, and therefore the solution algorithm is not unique too.

For convenience, we introduce the vector-valued functions { $\boldsymbol{y}_{i}(z)=\{y_{i,j}(z)\},$ $i=0,1,$ representing free variables in \eqref{eq:307:edge_initial_conditions}--\eqref{eq:307:edge_interelement_conditions} and defined on domains $(0,\tau_{0})$ and $(0,\tau_{1})$, respectively. For odd $N$, the components of $\boldsymbol{y}_{0}$ are 
$u_{n,2m}$ with $n \in J_{x} \setminus \{ -N, N \}$, $m = 1,2,\ldots,M-1$, and $w^{\pm}_{0,2m}$ with $m = 1,2,\ldots,M$ while the components of $\boldsymbol{y}_{1}$ are  $u_{n,2m+1}$ with $n \in J_{x} \setminus \{ -N, N \},$ $m = 1,2,\ldots,M-2$ and $w^{\pm}_{0,2m+1}$ with $m = 1,2,\ldots,M-1$. For even $N$, the components of $\boldsymbol{y}_{0}$ are $u_{n,2m}$ with $n \in J_{x} \setminus \{ -N, N \},$ $m = 1,2,\ldots,M -1$, $u_{0,2m+M}$ with $m=0,1$, and $w^{\pm}_{\mp 1,2m}$ with $m = 2,3,\ldots,M$, while $\boldsymbol{y}_{1}$ consists of $u_{n, 2m+1}$ with $n \in J_{x} \setminus \{ -N, N \}$, $m = 1,2,\ldots,M-2$, $u_{0,1+2m(M-1)}$ with $m=0,1$, and $w^{\pm}_{\mp 1,2m+1}$ with $m = 2,3,\ldots,M-1$. }

For any $N$, the unknowns $y_{i,j}$ are defined on the domain $z\in(0,\tau_{i})$, {$i=0,1.$} 
These functions are combined in vector-valued functions $\boldsymbol{y}_{0}(z)\in \mathbb{R}^{N^{0}_{s}}$ and  $\boldsymbol{y}_{1}(z)\in \mathbb{R}^{N^{1}_{s}}$. 
Here, $N^{0}_{s} = MN + M - N +1$, $N^{1}_{s} = MN + M - 2N$ both for odd and even $N$.
The total surplus number $N_{s} = N^{0}_{s} + N^{1}_{s}$ is given in~\eqref{eq:308:nonuniform_variable_surplus}.
In the proposed scheme for solving the system 
~\eqref{eq:307:edge_initial_conditions}--\eqref{eq:307:edge_interelement_conditions}, the free variables can be both traveling waves and 
control functions. Non-free control functions are their linear combinations. Therefore,  the  {vectors  $\boldsymbol{y}_{i}$ 
	ultimately determine the possible variation of  optimal control.} 


Similarly to a rod subject only to boundary controls \cite{KostinGavrikovJCSSI:2021, Kostin:2022a},    there exists a critical control time for the considered system. 

\begin{theorem}
If $T<2\lambda,$ the solution $(v, r)$ to the BVP \eqref{eq:301:boundary_value_problem} does not exist for arbitrary initial and terminal conditions {$v_{0},r_{0},v_{1},r_{1}\in H^{1}(I_{L})$}.
\end{theorem}

\begin{proof}
To prove that the system under study is not controllable if $T<2\lambda$ ($M=1$), we consider equations \eqref{eq:307:edge_interelement_conditions} for $M=1$. It  is enough to analyze only the  first equation in any of the systems~\eqref{eq:307:edge_interelement_conditions}. Indeed, {it} 
 has the form
	\begin{equation} \label{eq:309:unsatisfied_conditions} 
	\begin{array}{c}\ds
		w^{+}_{n-1,3} + w^{-}_{n-1,1} =
		w^{+}_{n+1,1} + w^{-}_{n+1,3}, \quad
		n\in J_{x} \setminus \{-N,N\},
	\end{array}
\end{equation}  
where $w^{+}_{n+1,1}$ $w^{-}_{n-1,1}$ are expressed through the given initial functions $v_{0}$, $r_{0}$, and $w^{+}_{n-1,3}$ $w^{-}_{n+1,3}$ depend on the {terminal} distributions $v_{1}$, $r_{1}$. Thus, the relations~\eqref{eq:309:unsatisfied_conditions} are fulfilled only with a special combination of initial and terminal states.  $\qed$
\end{proof}


Besides the continuity conditions  
 on the edges  
discussed in the previous subsections, the corresponding arrangement of unknowns at the mesh vertices has to be done.
To this end, we conjugate  with their neighbors only the free functions {$y_{i,j}(z)$  defined in Subsect.~\ref{sub:308}.} 

If $N$ is odd, 
the  conditions at vertices are given by
\begin{equation} \label{eq:310:odd_links}
	\begin{array}{l}
		u_{n,m}(0)=u_{n,m-1}(\tau_{1-i}), \quad
		m=2,3,\ldots,2M-1, \quad
		i=m\,\mathrm{mod}\, 2;
	\\
		w^{\pm}_{0,m}(0) = w^{\pm}_{0,m-1}(\tau_{1-i}),\quad  
		m =2,3,\ldots,2M+1, \quad
		i=m\,\mathrm{mod}\, 2, 	
	\end{array}
\end{equation}  
{where $n \in J_{x} \setminus \{ -N, N \}$.} Thus, the number of these equations  is ${N_b^1}= MN+M-N+1$ for $i=0$  and {$N^{0}_{b}=N^{1}_{b}$} for $i=1$. 

If $N$ is even, 
the  relations at vertices are given by
\begin{equation} \label{eq:310:even_links} 
	\begin{array}{l}
		u_{n,m}(0)=u_{n,m-1}(\tau_{1-i}),\quad
		m =2,3,\ldots,2M-1, \quad
		i=m\,\mathrm{mod}\, 2;
	\\
		u_{0,0}(0)=0,\quad
		u_{0,(2M-1)i+1}(0)=u_{0,(2M-1)i}(\tau_{i}),\quad
		i=0,1;   
	\\
		w^{\pm}_{\mp 1,m}(0)=w^{\pm}_{\mp 1,m-1}(\tau_{1-i}),\quad
		m=4,5,\ldots,2M+1,\quad
		i=m\,\mathrm{mod}\, 2, 
	\end{array}
\end{equation}  
{where $n \in J_{x} \setminus \{ -N, N \}$.}
For this case, the number of equations is equal to $N^{1}_b=MN+M-N+1$ for $i=0$ and $N^{0}_b=MN+M-N$ for $i=1$.

\section{Optimal Control Design}\label{sec:4}

		\subsection{Mean Energy Decomposition and One-Dimensional Variational Problem}\label{sub:401}
		
Let us analyze the structure of the objective functional $E$ in  terms of the traveling waves $w^{\pm}_{k,m}$.		

\begin{theorem}
	The OCP~\eqref{eq:301:homogeneous_control_problem}--\eqref{eq:301:homogeneous_ functionals} is reduced for $T>2\lambda$ to the variational problem: 
	Find such vector-valued functions  $\boldsymbol{y}_{i}^{*}(z)$ depending on one variable $z\in[0,\tau_{i}]$ $(i=0,1)$ and a constant $c^{*}_{1}\in\mathbb{R}$ that minimize a quadratic functional
		\begin{equation}\label{eq:402:variations_calculus_problem}
			\tilde E[\boldsymbol{y}_{0}^{*},\boldsymbol{y}_{1}^{*}] =
			\min\limits_{\boldsymbol{y}_{0},\boldsymbol{y}_{1},c_{1}} \tilde E[\boldsymbol{y}_{0},\boldsymbol{y}_{1}]
		\end{equation}	
	subject to  linear boundary constraints. 
\end{theorem}\label{Th:reduction to quadratic functional}

\begin{proof}
Paying attention to the structure of the energy density  $e$ defined in~\eqref{eq:301:homogeneous_ functionals}, and noticing that $r_t - f = v_x$ and $r_x = v_t$ on the solution, we get $e = \frac12 v_{t}^2 +\frac12 v_{x}^2$. By taking into account the expression for the displacements from~\eqref{eq:302:dAlembert_solutions}, the energy density {$e=e_k$} on each subdomain $\Omega_{k}$, $k\in J_{s}$, is given by	
    \begin{equation}\label{eq:401:energy_density_decomposition}
{        	\begin{array}{l}
e_k =
		\frac12\left(w^{+\prime}_{k}(z^{+}) + w^{-\prime}_{k}(z^{-})\right)^2 +
		\frac12\left(w^{+\prime}_{k}(z^{+}) - w^{-\prime}_{k}(z^{-})\right)^2 =
       \\
      	 e^{+}_{k}(z^{+}) +e^{-}_{k}(z^{-}), \quad      
        e^{+}_{k}(z^{+}) = \left(w^{+\prime}_{k}(z^{+})\right)^2,\quad 
        e^{-}_{k}(z^{-}) = \left(w^{-\prime}_{k}(z^{-})\right)^2.
	\end{array}
	}
    \end{equation}	
Here, {$(z^{+},z^{-}) \in \Omega_{k}$ and} the prime marks the derivatives of the functions $w^{\pm}_{k}$ with respect to {$z^\pm$.} 
As a result, the terms $e^{\pm}_{k}$ in the restriction $e_k$ of the energy density  $e$ on $\Omega_k$  depend only on one corresponding argument $z^{\pm}$. 

Since the functional $E$ of the mean mechanical energy is obtained from the function $e$ through a linear transformation (integration), it splits in its turn into independent parts $E_k^\pm$	
	\begin{equation*}
        	  E =\frac{1}{T}\int\nolimits_{\Omega}  e \,\mathrm{d}\Omega = 
        	  \sum\limits_{k\in J_s}\big(E^{+}_{k}+ E^{-}_{k}\big),\quad
			E^{\pm}_{k} = \frac{1}{T}\int\nolimits_{\Omega_{k}}  e^{\pm}_{k}(z^{\pm}) \,\mathrm{d}\Omega.
	\end{equation*}	
Here, the functional $E^{\pm}_{k}$ depends only on the traveling wave $w^{\pm}_{k}$ defined on their domains $I^{\pm}_{k}$~\eqref{eq:302:dAlembert_domains}, and the set of  indices $J_s$ is introduced in~\eqref{eq:203:space_indexing_sets}. 	
	
Substituting the expression for $e^{\pm}_{k}$ from~\eqref {eq:401:energy_density_decomposition} into $ E^{\pm}_{i}$, we arrive at 
	\begin{equation}\label{eq:401:term_mean_energy} 
         		E^{\pm}_{k} = \frac{1}{T}\int_{\Omega_{k}} \left(w^{\pm\prime}_{k}(z^{\pm})\right)^2  \,\mathrm{d}\Omega = 
        \frac{1}{T}\int_{z^{\pm}_{k}}^{T-z^{\mp}_{k}} \left(w^{\pm\prime}_{k}(z^{\pm})\right)^2 \Delta z^{\mp}_{k}(z^{\pm}) \,\mathrm{d}z^{\pm},
	\end{equation}	
Here, the piecewise linear functions
\begin{equation}\label{eq:delta z}
{       \Delta z^{\mp}_{k}(z^{\pm}) =
       \left\{
      	\begin{array}{ll}
             (z^{\pm} - z^{\pm}_{k}),  	& z^{\pm}\in[ z^{\pm}_{k}, -z^{\mp}_{k})
           \\[0 ex]
             \lambda,									&z^{\pm}\in[ -z^{\mp}_{k}, T+z^{\pm}_{k}]
           \\[0 ex]
             (T-z^{\mp}_{k}-z^{\pm}), 	&z^{\pm}\in(T+z^{\pm}_{k},T-z^{\mp}_{k}]
       	\end{array}
					\right. 
					}
\end{equation}
arise as a result of primary integration over the coordinate $z^{\pm}$ (details are available in the supplement).
%
Dividing the intervals of integration $I^{\pm}_{k}=[z^{\pm}_{k},T-z^{\mp}_{k}]$ {in~\eqref{eq:401:term_mean_energy}}
 into the subintervals $I^{\pm}_{k,m}$ and replacing the function $w^{\pm}_{k}$ with $w^{\pm}_{k,m}$ in accordance with~\eqref{eq:306:wave_function_intervals}, 
 we arrive at the expression of the mean energy
	\begin{equation}\label{eq:402:energy_decomposition}
		\begin{array}{c}\ds
			E = E^{+} + E^{-},\;
			E^{\pm} = \frac{1}{T}\sum\limits_{k\in J_{s}} 
			\sum\limits_{m\in J_{w}} 
			\int\limits_{0}^{\tau_{j_{m}}} \left(w^{\pm\prime}_{k,m}(z)\right)^2  \Delta z^{\mp}_{k}(z^{\pm}_{k,m}+z) \,\mathrm{d}z.
		\end{array}
	\end{equation}			
Here, the factors $\Delta z^{\pm}_{k}$ are the same as in~\eqref{eq:401:term_mean_energy}, 
 and $j_{m}=m\,\mathrm{mod}\, 2$.

		We introduce the vector-valued functions $\boldsymbol{w}_{i}:[0,\tau_{i}]\to \mathbb{R}^{N^{i}_{w}}$ with $i=0,1$ and $N^{i}_{w}=2N(M-i+2)$ through their elements {$w_{i,2j-1} = w^{+}_{k,2m+i},$ $w_{i,2j} = w^{-}_{k,2m+i},$ with $k\in J_{s},$ $2m+i\in J_{w},$ $j =(2M+3)(k+N-1)+m+1.$}			 
The positive definite diagonal  matrix-valued functions $\boldsymbol{G}^{i}:[0,\tau_{i}]\to {\mathbb{R}^{N_{w}^i\times N_{w}^i}}$ for $i=0,1$ are introduced via their non-zero elements 
 {$G^{i}_{2j-1,2j-1}(z) = $} {$\sqrt{\Delta z^{-}_{k}(z^{+}_{k,2m+i}+z)},$ $G^{i}_{2j,2j}(z)    = \sqrt{\Delta z^{+}_{k}(z^{-}_{k,2m+i}+z)},$ with $k\in J_{s},$ $2m+i\in J_{w},$ $j= (2M+3)(k+N-1)+m+1.$}       
Therefore, the cost functional $E$ is quadratic and can then be rewritten in the form
	\begin{equation}\label{eq:402:energy_vector_form}\displaystyle
			E = \tilde E[\boldsymbol{y}_{0},\boldsymbol{y}_{1}]=\frac{1}{T}\sum_{i=0}^{1} \int_{0}^{\tau_{i}} \varepsilon_i(z) \,\mathrm{d}z,\quad \varepsilon_i = \big(\boldsymbol{G}^{i}\boldsymbol{w}_{i}^{\prime}\big)\cdot \big(\boldsymbol{G}^{i} \boldsymbol{w}_{i}^{\prime}\big).
	\end{equation}			

The functions $\boldsymbol{w}_{i}$ are linearly expressed through {the 
functions $\boldsymbol{y}_{i}$ defined in Subsect.~\ref{sub:308}.} 
 By taking into account the initial states $\big(v_{0}(x),r_{0}(x)\big)$ and  the  terminal states $\big(v_{1}(x),r_{1}(x)\big)$, 
these relation are given by
	\begin{equation}\label{eq:402:linear_state_relation}
		\boldsymbol{w}_{i}(\boldsymbol{y}_{i}(z),z,c_{1}) = \boldsymbol{A}_{i}\boldsymbol{y}_{i}(z)  +\boldsymbol{g}_{i}(z)+ c_{1}\boldsymbol{a}_{i},\quad
		z \in [0,\tau_{i}],\quad
		i=0,1. 
	\end{equation}		
Here,   {$\boldsymbol{A}_{i} \in \mathbb{R}^{N^{i}_{w}\times N_{s}^i}$, $\boldsymbol{a}_{i} \in \mathbb{R}^{N^{i}_{w}}$ are known matrices and vectors}, and $\boldsymbol{g}_{i}:[0,\tau_{i}]\to \mathbb{R}^{N^{i}_{w}}$ are functions expressed through the initial and terminal values of $v$ and~$p$. 	

Finally, the vertex conditions either~\eqref{eq:310:odd_links} or~\eqref{eq:310:even_links} can be written as
	\begin{equation}\label{eq:402:essential_boundary_conditions}
		\begin{array}{c}\ds
{		\boldsymbol{B}_{1,i}\boldsymbol{y}_{i}(\tau_{i}) - \boldsymbol{B}_{0,1-i}\boldsymbol{y}_{1-i}(0) = 
		c_{1}\boldsymbol{b}_{1,i} + \boldsymbol{b}_{0,i}, }
		\\
{		\{\boldsymbol{B}_{0,i},\boldsymbol{B}_{1,i}\} \subset \mathbb{R}^{N^{i}_{b}\times N_{s}^i},\quad
		\{\boldsymbol{b}_{0,i},\boldsymbol{b}_{1,i}\} \subset \mathbb{R}^{N^{i}_{b}},\quad 		i=0,1.}
		\end{array}
	\end{equation}			
Therefore, to find solution to the OCP \eqref{eq:301:homogeneous_control_problem}--\eqref{eq:301:homogeneous_ functionals}, we need to minimize the quadratic functional \eqref{eq:402:energy_vector_form} subject to the linear boundary conditions \eqref{eq:402:essential_boundary_conditions}. $\qed$
\end{proof}


				\subsection{Solution of the One-Dimensional Variational Problem}\label{sub:403}
				
				\begin{theorem}
The solution to the one-dimensional minimization problem \eqref{eq:402:variations_calculus_problem}, \eqref{eq:402:energy_vector_form}--\eqref{eq:402:essential_boundary_conditions}  exists and unique, and can be found by solving a BVP for the linear ODE system with constant coefficients
{
		\begin{equation}\label{eq:402:Euler_Lagrange_ODE}
		\begin{array}{c}\ds
			\boldsymbol{p}^{\prime}_{i}(z)=0,\quad
				{\boldsymbol{p}_{i}:} =  \frac{\partial \varepsilon_i}{\partial \boldsymbol{y}^{\prime}_{i}} = \lambda
\boldsymbol{A}^{\mathrm{T}}_{i}\boldsymbol{A}_{i} {\boldsymbol{y}'_{i}(z)}
+ \lambda\boldsymbol{A}^{\mathrm{T}}_{i} {\boldsymbol{g}'_{i}(z)},\quad
			z\in[0,\tau_{i}],\quad
			i=0,1,
				\\
		\end{array}
	\end{equation}
	subject to the boundary conditions \eqref{eq:402:essential_boundary_conditions} and natural conditions
		\begin{equation} \label{eq:403:natural_boundary_conditions}
			\begin{array}{c}\ds
					\boldsymbol{p}_{i}(0) = \boldsymbol{C}^{\mathrm{T}}_{0,i} \boldsymbol{h}_{1-i},\quad
					\boldsymbol{p}_{i}(\tau_{i})  = \boldsymbol{C}^{\mathrm{T}}_{1,i}\boldsymbol{h}_{i},\quad
					\boldsymbol{h}_{i} \in \mathbb{R}^{N^{i}_{b}},
				\\[1ex]\ds
		\boldsymbol{C}_{0,i}=\boldsymbol{B}_{0,i}-\frac{\boldsymbol{b}_{1,1-i}\boldsymbol{b}^{\mathrm T}_{1,1-i}\boldsymbol{B}_{0,i}}{|\boldsymbol{b}_{1,1-i}|^2},\quad
		\boldsymbol{C}_{1,i}=\boldsymbol{B}_{1,i}-\frac{\boldsymbol{b}_{1,i}\boldsymbol{b}^{\mathrm T}_{1,i}\boldsymbol{B}_{1,i}}{|\boldsymbol{b}_{1,i}|^2}.			
				\end{array}
		\end{equation}				
Here, $\varepsilon_i$ are defined in~\eqref{eq:402:energy_vector_form}, \eqref{eq:402:linear_state_relation} and $\boldsymbol{h}_{i}$ are unknown Lagrange  multipliers. 	
}
\end{theorem}

\begin{proof}

Let us analyze the structure of the diagonal matrices $\boldsymbol{G}^{i}(z)$  in \eqref{eq:402:energy_vector_form}. Since elements of $\boldsymbol{G}^{i}(z)$  are expressed via the functions $\Delta z^\mp_k,$ we obtain that  non-constant entries of $\boldsymbol{G}^{i}(z)$ are only in rows corresponding to the functions $w^\pm_{k,m}(z)$ that are defined through fixed initial ($m=0,1$) and terminal ($m=2M+1, 2M+2$) conditions. 

Indeed, consider the values of the  functions $\Delta z^\mp_k$ in  \eqref{eq:delta z}.
If $m\notin \{0,1,2M+1,2M+2\}$, then it follows from \eqref{eq:306:wave_function_intervals} that the minimum of the argument $z^\pm_{k,m}+z$ of $\Delta z^\mp_k$ is equal to $z^\pm_{k,2}$, whereas the maximum is $z^\pm_{k,2M}+\tau_0$. Also due to \eqref{eq:306:wave_function_intervals}, $z^\pm_{k,2}=z^\pm_k+\lambda$ and $z^\pm_{k,2M}+\tau_0=z^\pm_k+\lambda M +\tau_0$. According to \eqref{eq:302:dAlembert_domains}, $z^\pm_k+\lambda=-z^\mp_k$, and by definition $\lambda M+\tau_0=T$. 
Thus, the argument $z^\pm_{k,m}+z$ of $\Delta z^\mp_k$ in \eqref{eq:402:energy_decomposition} for the chosen range of the index $m$ belongs to the interval $[-z^\pm_k,T+z^\pm_k]$. Then, it follows from \eqref{eq:401:term_mean_energy} that $\Delta z^\mp_k=\lambda$.
For $m\in \{ 0,1, 2M+1, 2M+2\}$, $\Delta z_k^\mp(z_{k,m}^\pm+z)$ are linear functions of $z$. These non-constant elements of $\boldsymbol{G}_{i}(z)$ are related to entries of $\boldsymbol{w}_i$ depending on initial and terminal conditions yielding zero variation. Therefore, these elements do not influence the variation of $\tilde E$ w.r.t. free variables $\boldsymbol{y}_{i}(z)$.

{ Due to \eqref{eq:402:energy_vector_form} and~\eqref{eq:402:linear_state_relation}, the conjugate to $\boldsymbol{y}_{i}$  variables, that is, vector-valued functions $\boldsymbol{p}_{i}$, are expressed 
as in~\eqref{eq:402:Euler_Lagrange_ODE}.}
Since the Lagrangians $\varepsilon_i$ depend on $\boldsymbol{y}'_{i}$ but not on $\boldsymbol{y}_{i}$, the terms $\frac{\partial \varepsilon_i}{\partial \boldsymbol{y}_{i}}$ do not appear in the Euler--Lagrange equations. 
Thus, the Euler--Lagrange ODEs with constant coefficients are given by~\eqref{eq:402:Euler_Lagrange_ODE}.

The existence of the solution to \eqref{eq:402:Euler_Lagrange_ODE} follows from strictly positive definiteness of the matrices $\boldsymbol{A}^{\mathrm{T}}_{i}\boldsymbol{A}_{i}$. Indeed, $\boldsymbol{A}^{\mathrm{T}}_{i}\boldsymbol{A}_{i}\geq 0$ by construction.
To show that $\boldsymbol{A}^{\mathrm{T}}_{i}\boldsymbol{A}_{i}> 0$, take $v_0=r_0=v_1=r_1=0$, then $\varepsilon_i=\lambda \boldsymbol{y}^{\prime\mathrm{T}}_i \boldsymbol{A}^{\mathrm{T}}_i \boldsymbol{A}_i \boldsymbol{y}^\prime_i$ since $\boldsymbol{g}_i=0$ and $c_1=0$ in \eqref{eq:402:linear_state_relation}. Consider the value of $\varepsilon_i$ on vectors of standard basis {in}  $\mathbb{R}^{N_s^i}$. To this end, take $\boldsymbol{y}^\prime_i =  \big(\delta_{j,l}\big)_{j=1}^{N_s^i}$, where $\delta_{j,l}$ is Kronecker delta. The non-zero component of $\boldsymbol{y}^\prime_i$ corresponds to either $w^{\pm\prime}_{k,m}$ or $u_{n,m}^\prime$. If $1$ is in place of $w^{\pm\prime}_{k,m}$ in $\boldsymbol{y}^\prime_i$, then $\varepsilon_i\ge \lambda \left(w^{\pm\prime}_{k,m}\right)^2 = \lambda>0$ 
{since $\varepsilon_i$  is quadratic form \eqref{eq:402:energy_vector_form} and  $\boldsymbol{G}_i$ is diagonal.}
If $1$ is in place of $u^\prime_{n,m}$ in $\boldsymbol{y}^\prime_i$, then due to \eqref{eq:307:edge_boundary_conditions} or \eqref{eq:307:edge_interelement_conditions}  at least one of the derivatives $w^{\pm\prime}_{k,m}$, $w^{\pm\prime}_{k,m+2}$ equals to some $c\neq 0$.
{Then} $\varepsilon_i\ge \lambda c^2>0$.
Thus, the matrix $\boldsymbol{A}^{\mathrm{T}}_i \boldsymbol{A}_i$ must be  positive for any basis vector.

 If $\boldsymbol{A}^{\mathrm{T}}_{i}\boldsymbol{A}_{i}$ are not strictly definite, then at least one of the components $y_{i,j}'$ of $\boldsymbol{y}'_i$ is absent in $\varepsilon_i$. According {to the definition of $\boldsymbol{y}_i$} 
 this component is either $w^{\pm\prime}_{k,m}$ or $u_{n,m}'$. If $y_{i,j}'=w^{\pm\prime}_{k,m},$ then the presence of $y_{i,j}'$ in $\varepsilon_i$ follows from diagonality  of $\boldsymbol{G}_i$. If $y_{i,j}'=u'_{n,m},$ then one of $w^{\pm\prime}_{k,m}$ is expressed through $u'_{n,m}$ due to \eqref{eq:307:edge_boundary_conditions} or \eqref{eq:307:edge_interelement_conditions}. Thus, each free variables $y_{i,j}'$ enters into   $\varepsilon_i$. Therefore,  $\boldsymbol{A}^{\mathrm{T}}_{i}\boldsymbol{A}_{i}>0$ {and $(\boldsymbol{A}^{\mathrm{T}}_{i}\boldsymbol{A}_{i})^{-1}$ exists. Then the solution to \eqref{eq:402:Euler_Lagrange_ODE} has the form} 
\begin{equation}\label{eq:solution}
\boldsymbol{y}_i=-\int_0^z\int_0^{\xi_2} (\boldsymbol{A}^{\mathrm{T}}_{i}\boldsymbol{A}_{i})^{-1}\boldsymbol{A}^{\mathrm{T}}_{i} {\boldsymbol{g}''_{i}(\xi_1)} 
 \,\mathrm{d}\xi_1\,\mathrm{d}\xi_2 +  \boldsymbol{\alpha}_i z  + \boldsymbol{\beta}_i (z-\tau_i), 
\end{equation}
where $\boldsymbol{\alpha_i}, \boldsymbol{\beta_i}\in \mathbb{R}^{N_s^i}.$ As follows from \eqref{eq:solution}, $\boldsymbol{y}_i$ belongs the same functional space as $\boldsymbol{g}_i$. In its turn,
 $\boldsymbol{g}_i$ is a linear combination of initial and terminal states $(v_0,r_0),$ $(v_1, r_1).$  If $(v_0,r_0),$ $(v_1, r_1)$ are from $H^1(-1, 1),$ then the solution $(v,r)$ to the OCP  expressed linearly via $\boldsymbol{y}_i$  is from $H^1(\Omega)$.

Then the {unknown vectors} $\boldsymbol{\alpha}_i$, $\boldsymbol{\beta}_i$ in \eqref{eq:solution} can be used to resolve boundary conditions \eqref{eq:402:essential_boundary_conditions}. Indeed, from \eqref{eq:310:odd_links}, \eqref{eq:310:even_links} follows that each value of $y_{i,j}(0)$ or $y_{i,j}(\tau_i)$  only enters one equation \eqref{eq:310:odd_links}, \eqref{eq:310:even_links}. Thus, we {would be} able to  resolve boundary conditions at $z=0$ through $\boldsymbol{\beta}_i$ and at $z=\tau_i$ through $\boldsymbol{\alpha}_i$. Therefore, the solution to the BVP  \eqref{eq:402:essential_boundary_conditions}, \eqref{eq:402:Euler_Lagrange_ODE}	 exists. Note that there are more constants than boundary conditions. Thus, some of them are undefined and are used for optimization based on transversality  conditions. The uniqueness of the solution follows from uniqueness of a minimum of a quadratic functional with a positive definite weight matrix \cite{Gelfand:2000}. 
$\qed$

As mentioned above, to actually find the solution of {\eqref{eq:402:variations_calculus_problem}, \eqref{eq:402:energy_vector_form}--\eqref{eq:402:essential_boundary_conditions}}, we need to  satisfy the transversality conditions 
		\begin{equation}\label{eq:403:boundary_transversality_conditions}
			\boldsymbol{p}_{0}(\tau_{0})\cdot \delta\boldsymbol{y}_{0}(\tau_{0}) -
			 \boldsymbol{p}_{0}(0)\cdot \delta\boldsymbol{y}_{0}(0) + 
			\boldsymbol{p}_{1}(\tau_{1})\cdot \delta\boldsymbol{y}_{1}(\tau_{1}) -
			 \boldsymbol{p}_{1}(0)\cdot \delta\boldsymbol{y}_{1}(0) = 0 .
	\end{equation}			
The variation of essential boundary constraints~\eqref{eq:402:essential_boundary_conditions} is given by
	\begin{equation}\label{eq:403:variation_boundary_conditions}
			\boldsymbol{B}_{1,i}\delta\boldsymbol{y}_{i}(\tau_{i}) - \boldsymbol{B}_{0,1-i}\delta\boldsymbol{y}_{1-i}(0) = 
		\delta c_{1}\boldsymbol{b}_{1,i},\quad
		i=0,1.
	\end{equation}			
By taken into account~\eqref{eq:403:variation_boundary_conditions}, we find that	
\begin{equation}\label{eq:602:variation_terminal_constant}
	\delta c_1=\frac{\boldsymbol{b}^{\mathrm T}_{1,i}}{|\boldsymbol{b}_{1,i}|^2}\left(\boldsymbol{B}_{1,i}\delta\boldsymbol{y}_i(\tau_i) - \boldsymbol{B}_{0,1-i}\delta\boldsymbol{y}_{1-i}(0)\right),\quad i=0,1.
	\end{equation}	
					
{After excluding} $\delta c_1$ from~\eqref{eq:403:variation_boundary_conditions} with the help of \eqref{eq:602:variation_terminal_constant} and comparing~\eqref{eq:403:variation_boundary_conditions} with~\eqref{eq:403:boundary_transversality_conditions}, {the 
 conditions} on the vectors $\boldsymbol{p}_{i}(0)$ and $\boldsymbol{p}_{i}(\tau_{i})$ are  represented {by~\eqref{eq:403:natural_boundary_conditions}.}
\end{proof}

 The existence of a solution to the Lagrange--Euler equation~\eqref{eq:402:Euler_Lagrange_ODE} follows from (i) the smoothness of the right-hand side in~\eqref{eq:402:Euler_Lagrange_ODE} and (ii) the fact that the number of vertex conditions at boundary points~\eqref{eq:402:essential_boundary_conditions} is certainly less or equal than the differential order of the system~\eqref{eq:402:Euler_Lagrange_ODE}. The uniqueness of the solution follows from the quadratic nature of the minimized functional~\eqref{eq:402:energy_vector_form} and its positive definiteness. 

{
There are two ways to find the optimal solution to the problem~\eqref{eq:301:homogeneous_control_problem},  \eqref{eq:301:boundary_value_problem} for the special values of time horizon: $T=M\lambda$ (that is $\tau_0=0$). 
The first way is to consider directly  the case $\tau_{0}=0$ and pose a uniform mesh on the time-space domain $\Omega$. 
The most important difference of this mesh from that described in Subsect.~\ref{sub:304} lies in the merging of the corresponding characteristics generated by the initial and terminal conditions (dashed and dotted} lines in Fig.~\ref{fig:02}).  
This means that only traveling waves $w^{\pm}_{k,m}$ in~\eqref{eq:306:edge_wave_functions} and control functions $u_{n,m}$ in~\eqref{eq:306:edge_control_functions} with odd indices $m$ remain in our approach. In this particular case, a simplified version the algorithm  discussed in Section~\ref{sec:3} can be applied, see details in \cite{Kostin:2022b}. All continuity conditions can be satisfied for $T=M\lambda\ge T^{*} = 2\lambda$. Thus, the value $T^{*}$ of the time horizon is the critical time for controllability of the system under study for $\tau_0=0$.


{
The second way to obtain the solution is to take the limit $\tau_{0}\to 0$ or $\tau_{1} = \lambda - \tau_{0}\to 0$ by using the property of continuous dependence of the solution on the parameter $\tau_{0}$. 
Since, as shown above, there is a unique limiting solution, there will be a continuous transition to this solution at $T=M\lambda$ for $M>2$ from the right ($\tau_0\to 0$) and left ($\tau_1\to 0$). 
If $M=2$ and $\tau_{0}\to 0$, the optimal motion $(v^{*}(t,x),r^{*}(t,x))$ as well as the control $\boldsymbol{u}^{*}(t)$ will approach to the critical solution at $T=T^{*}$.}

		\subsection{Example of the Optimal Control Design}\label{sub:404}

\begin{figure}[t]
	\begin{center}
          \noindent\centering{
          \large
          \unitlength \linewidth
          \begin{picture}(1,0.5)
          \put(0,0){\includegraphics[width=0.5\linewidth]{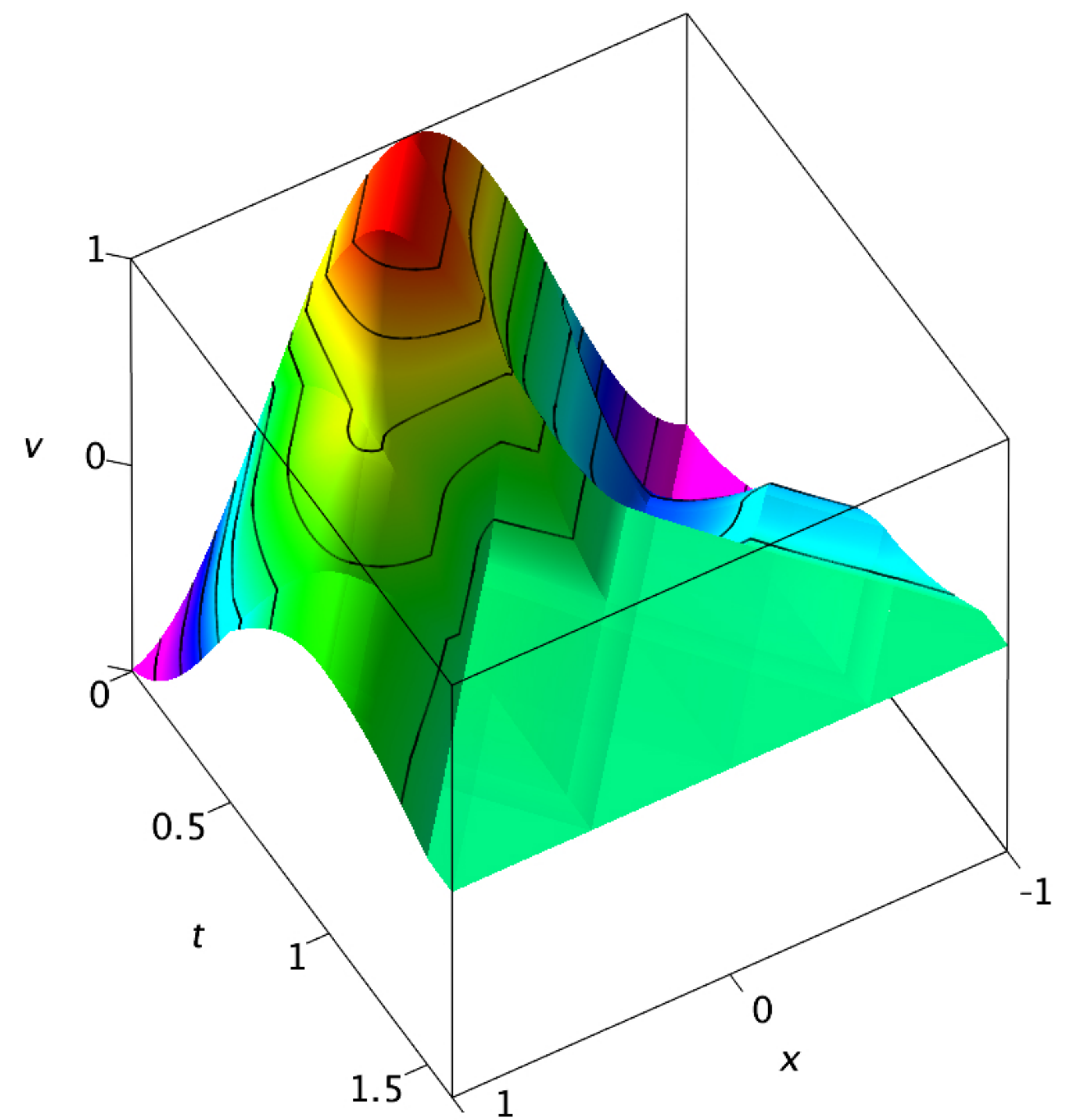}}
          \put(0.5,0){\includegraphics[width=0.5\linewidth]{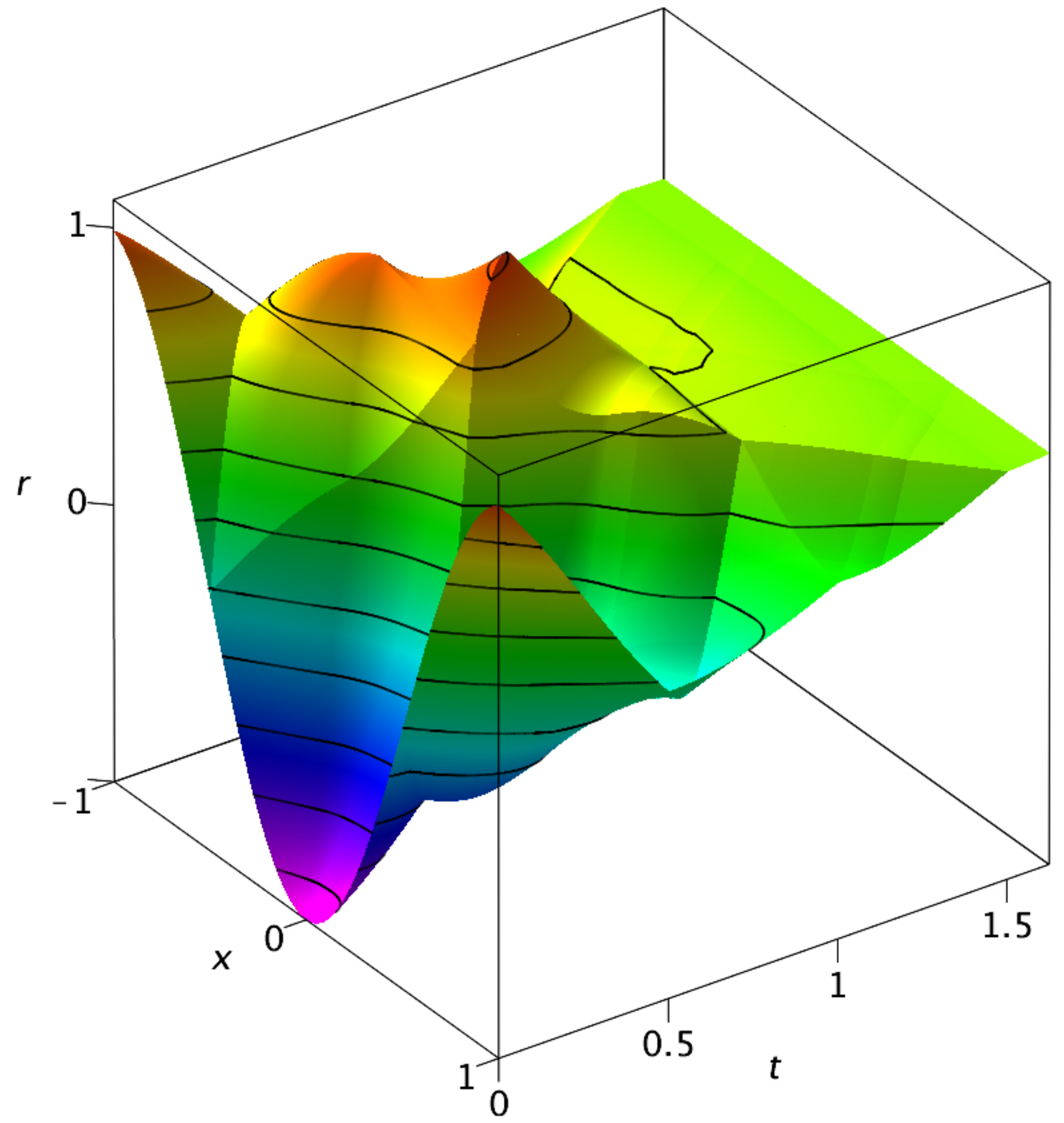}}          
          \put(0,0){\small (a)}
          \put(0.5,0){\small (b)}          
          \end{picture}\\
          }	
		\caption{{Optimal  solution for $N=4$, $T=1\frac{5}{8}$: (a) displacements $v(t,x)$, (b) potential $r(t,x)$.}}
		\label{fig:16}
	\end{center}
\end{figure}

Let us consider  the optimal control design,  that is Problem~\ref{prob:31}, for the system with four piezoelectric elements (see	Fig.~\ref{fig:02}).
The illustrative initial conditions throughout the subsection are	taken as	
{$v(0,x)=\cos 3x,$ $r(0,x)=-\cos 3x.$}
 Although the initial functions are quite simple and even, the solution to the problem has no symmetry about the time axis, since the distribution of the initial velocities of the rod points is an odd function of the spatial coordinate. 
Indeed, the initial momentum density is defined as $p_{0}(x)=r'_{0}(x)=3\sin 3x=v_{t}(0,x)$.
We take the terminal conditions	{$v(T,x)=0,$ $r(T,x)=c_1,$}
 which  means that the rod reaches its zero state at the end of the process.				
The control time is taken equal to $T = 1\frac{5}{8}$, what {generates} the mesh presented  in Fig.~\ref{fig:02}. The mesh parameters are $M=3$ $N=4$, $\lambda=\frac{1}{2}$, $\tau_0=\frac18$.

The resulting displacements $v(t,x)$ are shown in Fig.~\ref{fig:16}a. It is clearly seen that the rod reaches the 
{undeformed} state at the terminal time instant. The corner points appear along the characteristics,  which coincide with the edges of the mesh in Fig.~\ref{fig:02}.
In Fig.~\ref{fig:16}b the optimal dynamic potential $r(t,x)$ is presented. 
The terminal function $r(T,x)$ is constant and equal to $c_{1}  \approx  0.48$. 
 {Note that jumps} of derivatives occur on the same lines as for displacements.
Thus, the momentum density $p$ and force distribution $s$ defined by $r$ have discontinuities along the characteristics of the  mesh.
Nevertheless, the dynamic potential $r$ itself, according to~\eqref{eq:302:dAlembert_solutions}, is a continuous function.

%

The integrals of optimal force jumps $u_{n}(t)$ with $n \in J_{x} = \{-4,-2,0,2,4\}$ are shown in Fig.~\ref{fig:18}a.
These integrals are combinations of trigonometric and polynomial functions of time. Each $u_{n}(t)$ is continuous and by definition in~\eqref{eq:206:control_integrals} starts with the zero initial value $u_{n}(0)=0$.     
The optimal control forces $f_{n}(t)=u'_{n}(t)$ for $n\in J_{x}$ have jumps at the time instants $t=\frac18,\frac12,\frac58,1,\frac98,\frac32$. 
The control integrals $u_{k}(t)$ for the indices $k\in J_{c}=\{-5,-3,-1,1,3,5\}$ are continuous maps of the integrals of jumps $u_{n}(t)$ with $n \in \{-4,-2,0,2,4\}$ as shown in Subsect.~\ref{sub:0205}. 
Each function $u_{k}(t)$ has the zero initial value according to~\eqref{eq:206:control_integrals}. The values $u_{k}(T)$ for $k\in J_{s}=\{-3,1,1,3\}$ are derived from the terminal conditions~\eqref{eq:206:control_integrals} and 
{d'Alembert's} 
representation~\eqref{eq:302:dAlembert_solutions} of $r$. 
The terminal integrals $u_{\pm 5}(T)=c_{1}-r_{0}(\pm 1)$ are found via the boundary  and  terminal conditions~\eqref{eq:202:initial_conditions}, \eqref{eq:206:control_integrals}.
The original optimal control inputs $f_{k}(t)=u'_{k}(t)$ with $k \in J_{c}$ are obtained according to~\eqref{eq:204:control_relations} and presented for this example in Fig.~\ref{fig:18}b. 
The functions have discontinuities at the six time points   
and free values  at $t=0,T$.

\begin{figure}[t]
	\begin{center}
          \noindent\centering{
          \large
          \unitlength \linewidth
          \begin{picture}(1,0.5)
          \put(0,0){\includegraphics[width=0.4\linewidth]{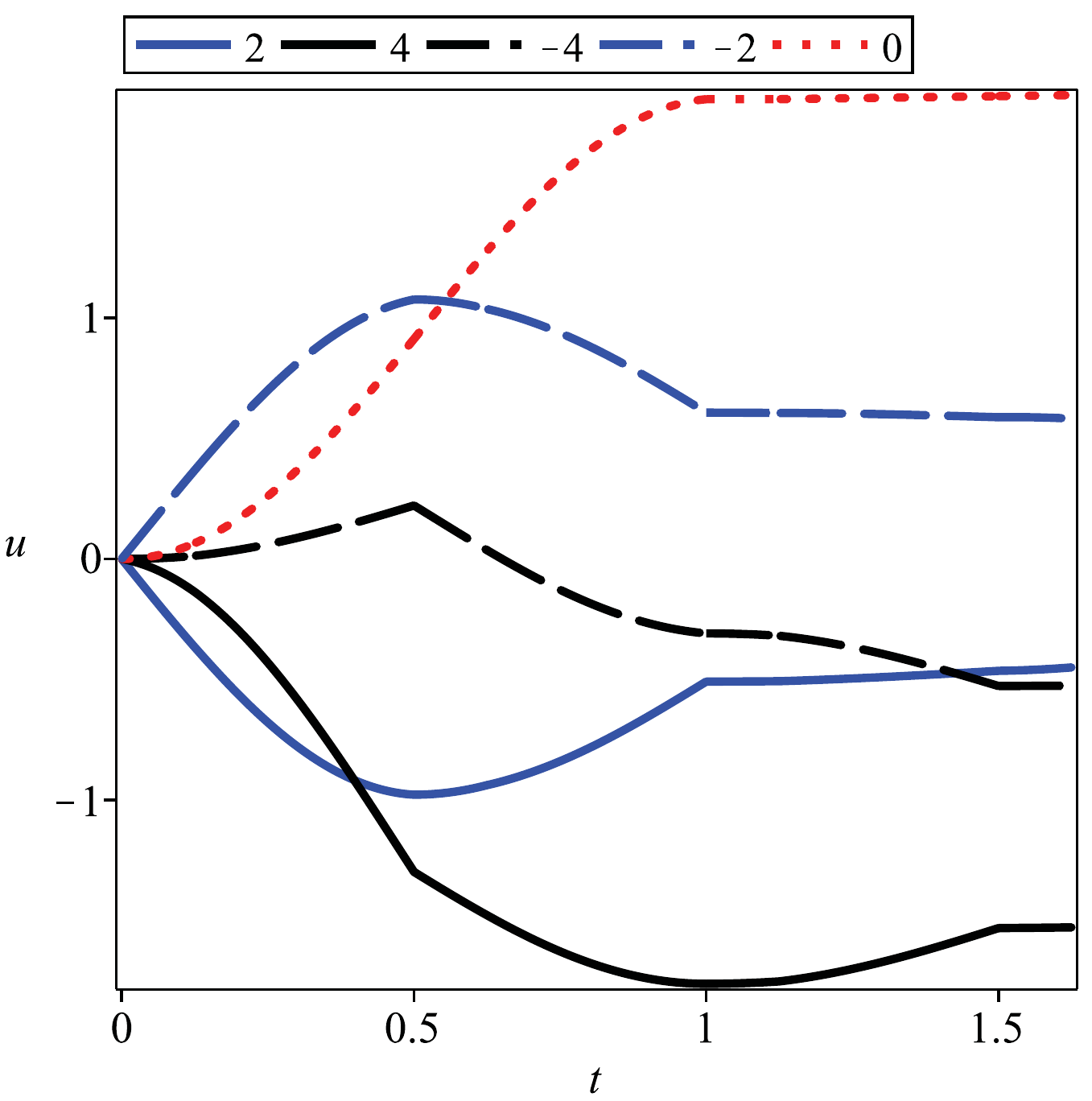}}
          \put(0.5,0){\includegraphics[width=0.4\linewidth]{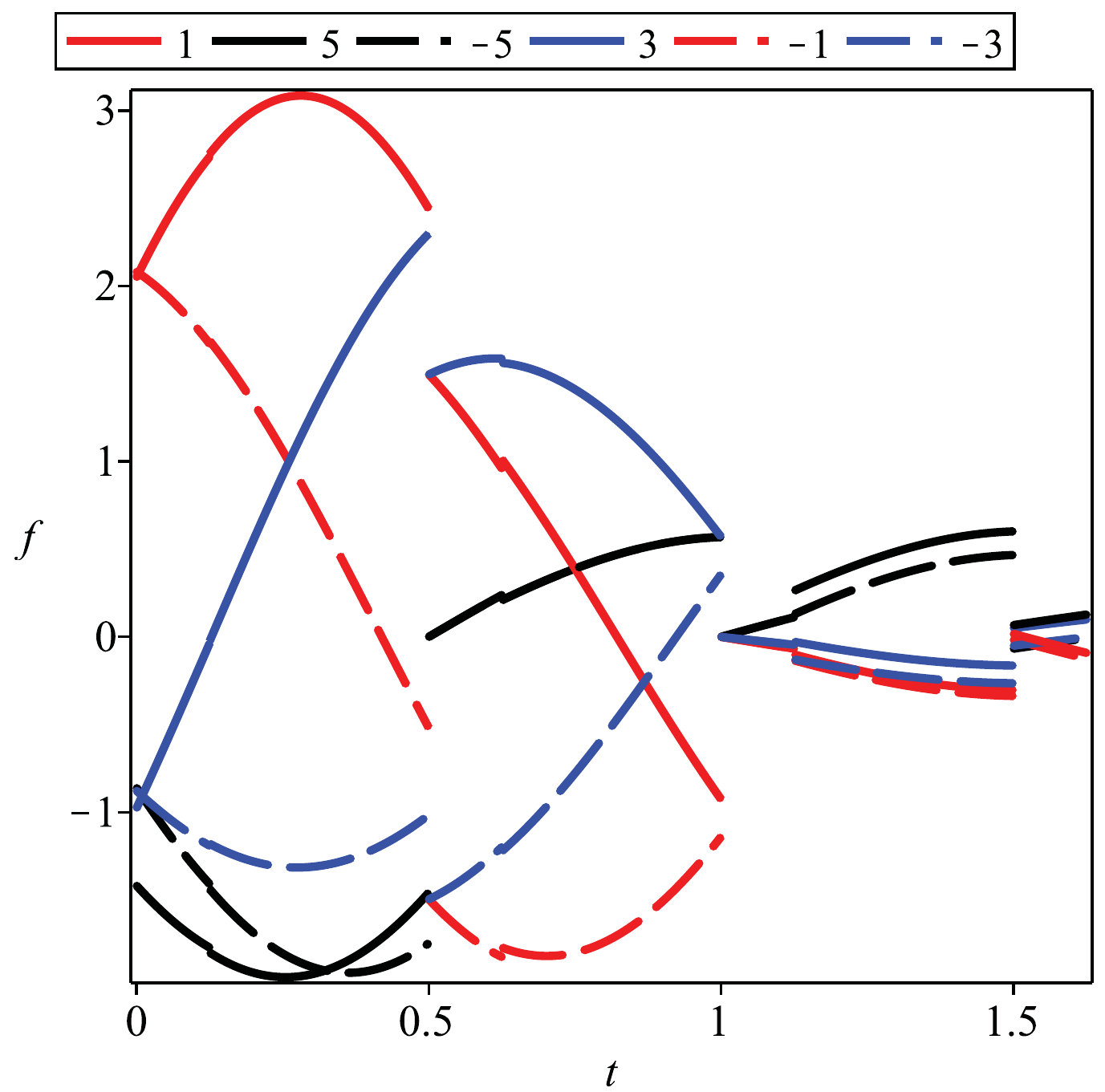}}          
          \put(0,0){\small (a)}
          \put(0.5,0){\small (b)}          
          \end{picture}\\
          }	
		\caption{{Optimal control inputs for $N=4$ and $T=1\frac{5}{8}$: (a) control $u_{n}(t)$ with the indices $n \in J_{x}$, (b) forces $f_{k}(t)$ with the indices $k \in J_{c}$.}}
		\label{fig:18}
	\end{center}
\end{figure}

		

		
{At the end, we present the dependency of the optimal value of the cost functional $E$ on the 
the control time $T$. 
We consider the integral of energy $F=T\cdot E(T,N)$ as a function of 
 $T>2\lambda$ for different numbers of the control elements $N$. 
 Due to the additivity of $F$, its value   certainly does not grow for increasing $T$ when the zero terminal state is considered.
The optimal values of the energy integral $F$ versus $T$ for $N=3,4,5,6$ (dot, dash, dashdot, {and} solid curves, respectively) for the initial and terminal states chosen above 
 are shown in~Fig.~\ref{fig:22}.
The optimal integral of energy for $N=2$ does not change with the control time $T$ and is equal to $F\approx 7.06$ for the chosen conditions and, thus, is omitted here. For $N>2$, 
$F(T,N)$ is continuous in $T$ and monotonically decreases when both $T$ and $N$ increase.
Except for the case $N=3$, the functions $F(T,N)$ are convex in $T$ on the open intervals of their smoothness $T\in (2M/N,2(M+1)/N)$.
The controlability condition discussed above restricts the domain of the map $F$ so that the control time $T\in[4/N,+\infty]$.  
As seen in Fig.~\ref{fig:22}, the rate of decrease of the functional $F$ for any $N$  almost vanishes when $T>2$.}   

\begin{figure}[t]
	\begin{center}
		\includegraphics[width=0.4\linewidth]{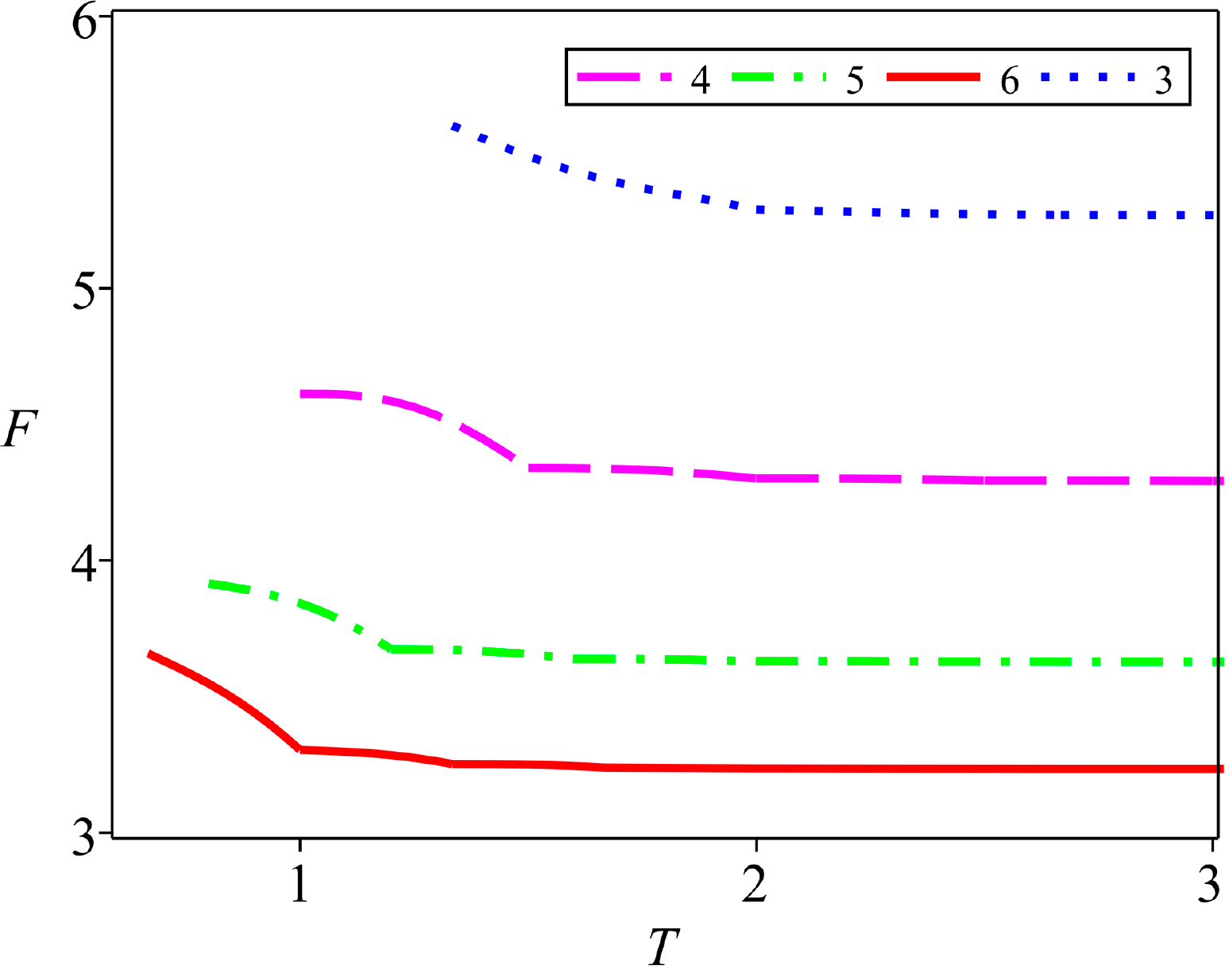}\\ 
		\caption{Optimal integral of energy $F$ vs. control time $T$ for $N=3,4,5,6$.}
		\label{fig:22}
	\end{center}
\end{figure}

\section{Conclusions and Outlook}\label{sec:5}

 The motion of a dynamic system
under external boundary loads and internal distributed stresses has been studied. 
The proposed mathematical model can be related to longitudinal vibrations of a thin elastic rod controlled by piezoelectric actuators symmetrically attached along its central line together with normal  forces at the ends. 	 
Since most real-world implementations of dynamical systems necessarily involve discretization,  we study rigorously a problem that is already discretized with respect to the distributed control input while our solution algorithm does not require discretization of state variables. 
For given initial and terminal states and a fixed time horizon, the optimal control problem is to minimize the mean energy stored in the rod during the motion.  
In the case of equidistantly placed actuators and a uniform rod, the shortest possible time for bringing the system with a given number of control elements to an arbitrary state is defined. 
An optimization algorithm using traveling waves is proposed to reduce the original problem to a one-dimensional variational problem with boundary conditions of a special kind.

We plan to study further the controllability of this system in the absence of external boundary forces and in the case when some of the piezoelectric elements are turned off or they are equidistantly spaced apart. 
We are also interested in {designing}  a bounded control and in estimating the accuracy of finite-mode approximations. 
This makes possible to look at more realistic models and propose a feedback on-line control.
In this regard, the issues of observability of a system with distributed piezoelectric sensors are relevant. 
A  possible extension of the proposed approach is to consider the problem of an elastic rod's bending and to develop numerical optimization procedures for inhomogeneous rods.  From a theoretical point of view, the transition to the limit case of an infinite number of infinitesimal actuators may show what assumptions on an infinite-dimensional (continuous in space) distributed input should be imposed to adequately exploit discretization for a practical use of such an input.

\bibliographystyle{ieeetr}
\bibliography{references}
                          
\end{document}